\documentclass[review,3p,square,11pt]{elsarticle}

\usepackage{amssymb}
\usepackage{amsmath}
\usepackage{amsthm}
\usepackage{bbm}
\usepackage{enumitem}

\usepackage{fancyhdr}
\usepackage{hyperref}
\hypersetup{%
  pdftitle={BSDEs},%
  pdfsubject={BSDEs},%
  pdfauthor={ShengJun FAN, Ying Hu},%
  pdfkeywords={BSDEs},%
  pdfstartview=FitH,%
  CJKbookmarks=true,%
  bookmarksnumbered=true,%
  bookmarksopen=true,%
  colorlinks=true, linkcolor=blue, urlcolor=blue, citecolor=blue, %
}

\usepackage{cleveref} 
\crefname{thm}{Theorem}{Theorems}
\crefname{pro}{Proposition}{Propositions}
\crefname{lem}{Lemma}{Lemmas}
\crefname{rmk}{Remark}{Remarks}
\crefname{cor}{Corollary}{Corollaries}
\crefname{dfn}{Definition}{Definitions}
\crefname{ex}{Example}{Examples}
\crefname{section}{Section}{Sections}
\crefname{subsection}{Subsection}{Subsections}

\newcommand{\eps}{\varepsilon}

\newcommand{\To}{\rightarrow}
\newcommand{\as}{{\rm d}\mathbb{P}\times{\rm d} t-a.e.}
\newcommand{\ass}{{\rm d}\mathbb{P}\times{\rm d} s-a.e.}
\newcommand{\ps}{\mathbb{P}-a.s.}

\newcommand{\F}{\mathcal{F}}
\newcommand{\E}{\mathbb{E}}

\newcommand{\s}{\mathcal{S}}
\newcommand{\lcal}{\mathcal{L}}
\newcommand{\mcal}{\mathcal{M}}

\newcommand{\T}{[0,T]}

\newcommand{\R}{{\mathbb R}}

\newcommand{\RE}{\forall}

\newcommand {\Lim}{\lim\limits_{n\rightarrow\infty}}
\newcommand {\Dis}{\displaystyle}

\newtheorem{thm}{Theorem}[section]

\newtheorem{pro}[thm]{Proposition}
\newtheorem{rmk}[thm]{Remark}

\newtheorem{dfn}[thm]{Definition}
\newtheorem{ex}[thm]{Example}

\journal{}
\begin{document}
\begin{frontmatter}

\title{{Existence, uniqueness, comparison theorem and stability theorem for unbounded solutions of scalar BSDEs with sub-quadratic generators}\tnoteref{found}}
\tnotetext[found]{Shengjun Fan is supported by the State Scholarship Fund from the China Scholarship Council (No. 201806425013). Ying Hu is partially supported by Lebesgue center of mathematics ``Investissements d'avenir" program-ANR-11-LABX-0020-01, by CAESARS-ANR-15-CE05-0024 and by MFG-ANR-16-CE40-0015-01.\vspace{0.2cm}}


\author{Shengjun Fan\vspace{-0.5cm}\corref{cor1}}
\author{\ \ \ \ Ying Hu\corref{cor2}}
\cortext[cor1]{School of Mathematics, China University of Mining and Technology, Xuzhou 221116, China. E-mail: f\_s\_j@126.com \vspace{0.2cm}}

\cortext[cor2]{Univ. Rennes, CNRS, IRMAR-UMR6625, F-35000, Rennes, France. E-mail: ying.hu@univ-rennes1.fr}

\begin{abstract}
We first establish the existence of an unbounded solution to a backward stochastic differential equation (BSDE) with generator $g$ allowing a general growth in the state variable $y$ and a sub-quadratic growth in the state variable $z$, like $|z|^\alpha$ for some $\alpha\in (1,2)$, when the terminal condition satisfies a sub-exponential moment integrability condition like $\exp\left(\mu L^{2/\alpha^*}\right)$ for the conjugate $\alpha^*$ of $\alpha$ and a positive parameter $\mu>\mu_0$ with a certain value $\mu_0$, which is clearly weaker than the usual $\exp(\mu L)$ integrability and stronger than $L^p\ (p>1)$ integrability. Then, we prove the uniqueness and comparison theorem for the unbounded solutions of the preceding BSDEs under the additional assumptions that the terminal conditions have sub-exponential moments of any order and the generators are convex or concave in $(y,z)$. Afterwards, we extend the uniqueness and comparison theorem to the non-convexity and non-concavity case, and establish a general stability result for the unbounded solutions of the preceding BSDEs. Finally, with these tools in hands, we derive the nonlinear Feynman-Kac formula in this context.\vspace{0.2cm}
\end{abstract}

\begin{keyword}
Backward stochastic differential equation \sep Existence and uniqueness \sep \\ \hspace*{2cm} Sub-quadratic growth \sep $\exp\left(\mu L^{2 /\alpha^*}\right)$-integrability \sep Comparison theorem \sep \\
\hspace*{2cm} \vspace{0.2cm}Stability theorem \sep Feynman-Kac formula.

\MSC[2010] 60H10\vspace{0.2cm}
\end{keyword}

\end{frontmatter}
\vspace{-0.4cm}

\section{Notations and introduction}
\label{sec:1-Introduction}
\setcounter{equation}{0}

Let us fix a positive real number $T>0$ and a positive integer $d$, and let $x\cdot y$ represent the usual scalar inner product for $x,y\in \R^d$. Let $(B_t)_{t\in\T}$ be a standard $\R^d$-valued Brownian motion defined on some complete probability space $(\Omega, \F, \mathbb{P})$ with $(\F_t)_{t\in\T}$ being its natural filtration augmented by all $\mathbb{P}$-null sets of $\F$. All the measurability with respect to processes will refer to this filtration. Let us recall that a progressively measurable scalar process $(X_t)_{t\in\T}$ belongs to class (D) if the family of random variables $\{X_\tau\}$, $\tau$ running all $(\F_t)$-stopping times valued in $\T$, is uniformly integrable.

Denote $\R_+:=[0,+\infty)$, ${\bf 1}_A(x)=1$ when $x\in A$ otherwise $0$, and ${\rm sgn}(x):={\bf 1}_{x>0}-{\bf 1}_{x\leq 0}$. Let $a\wedge b$ denote the minimum of two real $a$ and $b$, $a^-:=-(a\wedge 0)$ and $a^+:=(-a)^-$. For each $\alpha\in (1,2)$, let $\alpha^*$ stand for the conjugate of $\alpha$, that is, $1/\alpha+1/\alpha^*=1$ or
$$
\alpha^*:=\frac{\alpha}{\alpha-1}>2.
$$
For any real $p\geq 1$, let $L^p$ represent the set of (equivalent classes of) all $\F_T$-measurable scalar random variables $\xi$ such that $\E[|\xi|^p]<+\infty$, $\lcal^p$ the set of (equivalent classes of) all progressively measurable scalar processes $(X_t)_{t\in\T}$ such that
$$
\|X\|_{\lcal^p}:=\left\{\E\left[\left(\int_0^T |X_t|{\rm d}t\right)^p\right]\right\}^{1/p}<+\infty,\vspace{0.1cm}
$$
$\s^p$ the set of (equivalent classes of) all progressively measurable and continuous scalar processes $(Y_t)_{t\in\T}$ such that
$$\|Y\|_{{\s}^p}:=\left(\E[\sup_{t\in\T} |Y_t|^p]\right)^{1/p}<+\infty,\vspace{0.1cm}$$
and $\mcal^p$ the set of (equivalent classes of) all progressively measurable $\R^d$-valued processes $(Z_t)_{t\in\T}$ such that
$$
\|Z\|_{\mcal^p}:=\left\{\E\left[\left(\int_0^T |Z_t|^2{\rm d}t\right)^{p/2}\right] \right\}^{1/p}<+\infty.
$$

We study the following backward stochastic differential equation (BSDE for short):
\begin{equation}\label{eq:1.1}
  Y_t=\xi+\int_t^T g(s,Y_s,Z_s){\rm d}s-\int_t^T Z_s \cdot {\rm d}B_s, \ \ t\in\T,
\end{equation}
where $\xi$ is an $\F_T$-measurable scalar random variable called the terminal condition, the function
$g(\omega, t, y, z):\Omega\times\T\times\R\times\R^d \mapsto \R $
is progressively measurable for each $(y,z)$ and continuous in $(y,z)$ called the generator, and the pair of processes $(Y_t,Z_t)_{t\in\T}$ valued in $\R\times\R^d$ is called the solution of eq. \eqref{eq:1.1}, which is progressively measurable such that $\ps$, $t\mapsto Y_t$ is continuous, $t\mapsto Z_t$ is square-integrable, $t\mapsto g(t,Y_t,Z_t)$ is integrable, and verifies \eqref{eq:1.1}.

In this paper, we always assume that $\xi$ is a terminal condition and $g$ is a generator which is continuous in $(y,z)$, and we use BSDE$(\xi,g)$ to denote the BSDE with terminal condition $\xi$ and generator $g$. We consider the BSDE with generator $g$ satisfying $\as$,
\begin{equation}\label{eq:1.2}
\RE\ (y,z)\in \R\times\R^d,\ \ \  |g(\omega,t,y,z)|\leq |g(\omega,t,0,0)|+\beta|y|+\gamma |z|^\alpha
\end{equation}
with $\alpha>0$, $\beta\geq 0$ and $\gamma>0$. We usually say that $g$ has a linear growth in the state variable $z$ when $\alpha=1$, a sub-linear growth in $z$ when $\alpha\in (0,1)$, a quadratic growth in $z$ when $\alpha=2$, a superquadratic growth in $z$ when $\alpha>2$, and a sub-quadratic growth in $z$ when $\alpha\in (1,2)$. Our attention focuses on the last case. Let us first recall some related results in previous four cases, which have been intensively studied. For narrative convenience, we denote $g_0:=g(\cdot,0,0)$.

Assume first that the generator $g$ has a linear growth in $(y,z)$, i.e., \eqref{eq:1.2} with $\alpha=1$ holds for $g$. It is well known that for $(\xi,g_0)\in L^p\times\lcal^p$ with some $p>1$, BSDE$(\xi,g)$ admits a solution in  $\s^p\times\mcal^p$, and the solution is unique when $g$ further satisfies the uniformly Lipschitz condition in $(y,z)$. Readers are referred to \citet{PardouxPeng1990SCL,ElKarouiPengQuenez1997MF,LepeltierSanMartin1997SPL,
BriandDelyonHu2003SPA} and \citet{FanJiang2012JAMC} for more details. Recently,  \citet{HuTang2018ECP,BuckdahnHuTang2018ECP} and \citet{FanHu2019ECP} extended this result and established the existence and uniqueness of an unbounded solution to BSDE$(\xi,g)$ with linear-growth generator $g$ by assuming that $(\xi,g_0)$ satisfies an $L\exp\left(\mu\sqrt{2\ln(1+L)}\right)$-integrability condition for $\mu\geq \gamma \sqrt{T}$, which is weaker than the usual $L^p\ (p>1)$ integrability and stronger than $L\ln L$ integrability.

Secondly, assume that the generator $g$ has a linear growth in $y$ and a sub-linear growth in $z$, i.e., eq. \eqref{eq:1.2} with $\alpha\in (0,1)$ is satisfied for $g$. It follows from \citet{BriandDelyonHu2003SPA} that for $(\xi,g_0)\in L^1\times\lcal^1$, BSDE$(\xi,g)$ admits a solution $(Y_\cdot,Z_\cdot)$ such that $Y_\cdot$ belongs to class (D), and the solution is unique when $g$ further satisfies the uniformly Lipschitz condition in $(y,z)$. See for example \citet{BriandHu2006PTRF} and \citet{Fan2016SPA,Fan2018JOTP} for more details on this topic.

Thirdly, from \citet{DelbaenHuBao2011PTRF} it is well known that superquadratic
BSDEs, i.e., \cref{eq:1.2} with $\alpha>2$ holds for the generator $g$, are not solvable in general.
Some solvability results under the Markovian setting can be founded in \citet{DelbaenHuBao2011PTRF},  \citet{MasieroRichou2013EJP}, \citet{Richou2012SPA} and \citet{CheriditoNam2014JFA}.

Finally, we assume that the generator $g$ has a linear growth in $y$ and a quadratic growth in $z$, i.e., eq. \eqref{eq:1.2} with $\alpha=2$ is satisfied for $g$. It is also well known from \citet{Kobylanski2000AP} that if both $\xi$ and $\int_0^T |g_0| {\rm d}t$ are bounded, then BSDE$(\xi,g)$ admits a solution $(Y_\cdot,Z_\cdot)$ such that $Y_\cdot$ is a bounded process and $Z_\cdot\in \mcal^2$, and the solution is unique if $g$ further satisfies the uniformly Lipschitz condition in $y$ and a locally Lipschitz condition in $z$. Readers are referred to  \citet{BriandElie2013SPA,HuTang2016SPA} and \citet{Fan2016SPA} for further research on the bounded solution of quadratic BSDEs. Later, \citet{BriandHu2006PTRF,BriandHu2008PTRF} and \citet{DelbaenHuRichou2011AIHPPS} extended this result and established the existence and uniqueness of an unbounded solution to BSDE$(\xi,g)$ with generator $g$ having a quadratic growth in $z$ by assuming that $(\xi,g_0)$ has only $\gamma e^{\beta T}$-order exponential moment integrability, where the generator $g$ need to be uniformly Lipschitz with respect to the variable $y$ and convex (concave) with respect to the variable $z$ for the uniqueness of the solution, see also \citet{BarrieuKaroui2013AoP} for more details.

In this paper, we study the existence, uniqueness, comparison theorem and stability theorem for unbounded solutions of BSDE$(\xi,g)$ with generator $g$ having a linear growth in $y$ and a sub-quadratic growth in $z$, namely, eq. \eqref{eq:1.2} with $\alpha\in (1,2)$ is satisfied for $g$. We prove that if $(\xi,g_0)$ satisfies an $\exp\left(\mu L^{2 /\alpha^*}\right)$-integrability condition for a positive parameter $\mu>\mu_0$ with a certain value $\mu_0$, which is clearly weaker than the usual $\exp(\mu L)$ integrability and stronger than $L^p\ (p>1)$ integrability, then BSDE$(\xi,g)$ admits a solution $(Y_\cdot,Z_\cdot)$ such that $Y_\cdot$ belongs to class (D) and $Z_\cdot\in \mcal^2$, and the solution is unique and the comparison theorem and stability theorem hold when $g$ further satisfies a (extended) convexity or concavity condition with respect to the variables $(y,z)$, and $(\xi,g_0)$ satisfies the $\exp\left(\mu L^{2/ \alpha^*}\right)$-integrability condition for all $\mu>0$. We remark that in our final results, the linear growth condition of the generator $g$ in $y$ is also weakened to a one-sided growth condition, see \ref{H1"} in \cref{rmk:3.5} at the end of \cref{sec:3-Main results}.

The paper is organized as follows. In the next section, we introduce the whole idea of this paper and make an important preparation (see \cref{Pro:2.1}) for the proof of the main results. In Section 3 we establish the existence result and in Section 4 we prove the uniqueness and comparison theorem under the convexity or concavity condition of the generator. Afterwards, in Section 5 we extend the uniqueness and comparison theorem to the non-convexity and non-concavity case, and in Section 6 we establish a stability result for unbounded solutions for the preceding BSDEs under general assumptions. Finally, with these tools in hands, in Section 7 we derive the nonlinear Feynman-Kac formula in this context.\vspace{-0.1cm}

\section{The whole idea}
\label{sec:2-Whole idea}
\setcounter{equation}{0}

For the existence of an unbounded solution to BSDE$(\xi,g)$ with generator $g$ satisfying \eqref{eq:1.2} with $\alpha\in (1,2)$, our whole strategy is to establish some uniform a priori estimate on the first process $Y^{n,p}_\cdot$ in the solution of the usual approximated BSDEs (see the definition in \eqref{eq:3.11} of \cref{sec:3-Main results}) and to apply the localization procedure put forward initially in \citet{BriandHu2006PTRF}. In order to obtain the a priori estimate, the idea consists in searching for an appropriate function $\phi(s,x)$ and applying It\^{o}-Tanaka's formula to $\phi(s,|Y^{n,p}_s|)$ on the time interval $s\in [t,\tau_m]$ with $(\F_t)$-stopping time $\tau_m$ valued in $[t,T]$ for $t\in \T$. More precisely, we need to find a positive real number $\delta>0$, and a positive and smooth function $\phi(s,x):\T\times\R_+\To \R_+$ satisfying that $\phi_x(s,x)>0$, $\phi_{xx}(s,x)>\delta$ and
\begin{equation}\label{eq:2.1}
\begin{array}{c}
\Dis -\phi_x(s,x) \left(\beta x+\gamma |z|^\alpha\right) +{1\over 2}(\phi_{xx}(s,x)-\delta)|z|^2+\phi_s(s,x)\geq 0, \vspace{0.1cm}\\
\Dis \ (s,x,z)\in \T\times \R_+\times\R^d,
\end{array}
\end{equation}
where and hereafter, $\phi_s(\cdot,\cdot)$ stands for the first-order partial derivative of $\phi(\cdot,\cdot)$ with respect to the first variable, and $\phi_x(\cdot,\cdot)$ and $\phi_{xx}(\cdot,\cdot)$ respectively the first-order and second order partial derivative of $\phi(\cdot,\cdot)$ with respect to the second variable.

Observe from Young's inequality that
$$
{\gamma\phi_x(s,x)\over \phi_{xx}(s,x)-\delta} |z|^\alpha\leq c_{\alpha,\gamma} \left({\phi_x(s,x)\over \phi_{xx}(s,x)-\delta}\right)^{\frac{2}{2-\alpha}}+\frac{1}{2}|z|^2
$$
and then
$$
\begin{array}{lll}
&\Dis -\gamma \phi_x(s,x)|z|^\alpha+{1\over 2}(\phi_{xx}(s,x)-\delta)|z|^2 \vspace{0.3cm}\\
=& \Dis (\phi_{xx}(s,x)-\delta)\left(-{\gamma\phi_x(s,x)\over \phi_{xx}(s,x)-\delta}|z|^\alpha+{1\over 2}|z|^2\right)\vspace{0.3cm}\\
\geq & \Dis -c_{\alpha,\gamma}{\left(\phi_x(s,x)\right)^{\frac{2}{2-\alpha}}\over \left(\phi_{xx}(s,x)-\delta\right)^{\frac\alpha{2-\alpha}}},
\end{array}
$$
where
$$
c_{\alpha,\gamma}:=\frac{2-\alpha}{2}\alpha^{\frac\alpha{2-\alpha}}
\gamma^{\frac{2}{2-\alpha}}.
$$
It is clear that \eqref{eq:2.1} holds if the function $\phi(\cdot,\cdot)$ satisfies the following condition:
\begin{equation}\label{eq:2.2}
-\beta\phi_x(s,x) x -c_{\alpha,\gamma}{\left(\phi_x(s,x)\right)^{\frac{2}{2-\alpha}}\over \left(\phi_{xx}(s,x)-\delta\right)^{\frac\alpha{2-\alpha}}}+\phi_s(s,x)\geq 0,\ \ (s,x)\in \T\times\R_+.\vspace{0.2cm}
\end{equation}

Now, let $\mu_s:\T\To\R_+$ be a nondecreasing and continuously differentiable function with $\mu_0=\eps$ for some $\eps>0$, and let\vspace{-0.1cm}
\begin{equation}\label{eq:2.3}
k_{\alpha,\eps}:=\left(\frac{(1+\eps)^{\frac{2-\alpha}\alpha}}{2(\alpha-1) \eps \left((1+\eps)^{\frac{2-\alpha}\alpha}-1 \right)} \right)^{\frac\alpha{2(\alpha-1)}}.
\end{equation}
We choose the following function
\begin{equation}\label{eq:2.4}
\phi(s,x;\eps):=\exp\left(\mu_s \left(x+k_{\alpha,\eps}\right)^{\frac{2}{\alpha^*}}\right)=\exp\left(\mu_s \left(x+k_{\alpha,\eps}\right)^{\frac{2(\alpha-1)}{\alpha}}\right), \ \ \ (s,x)\in \T\times \R_+
\end{equation}
to explicitly solve the inequality \eqref{eq:2.2}. For each $(s,x)\in \T\times \R_+$, a simple computation gives
\begin{equation}\label{eq:2.5}
\phi_x(s,x;\eps)=\phi(s,x;\eps)\frac{2(\alpha-1)\mu_s}{\alpha(x+k_{\alpha,\eps})^{\frac{2-\alpha}
\alpha}}>0,
\end{equation}
\begin{equation}\label{eq:2.6}
\phi_{xx}(s,x;\eps)=\phi(s,x;\eps)\frac{2(\alpha-1)\mu_s
\left[2(\alpha-1)\mu_s(x+k_{\alpha,\eps})^{
\frac{2(\alpha-1)}\alpha}-1+(\alpha-1)\right]}{\alpha^2(x+k_{\alpha,\eps})^{
\frac{2}\alpha}}>0
\end{equation}
and
\begin{equation}\label{eq:2.7}
\phi_s(s,x;\eps)=\phi(s,x;\eps)(x+k_{\alpha,\eps})^{\frac{2(\alpha-1)}
\alpha}\mu'_s>0.\vspace{0.1cm}
\end{equation}
Furthermore, for each $(s,x)\in \T\times \R_+$, in view of the fact of $\mu_s\geq \mu_0=\eps$ and \eqref{eq:2.3}, we have
\begin{equation}\label{eq:2.8}
2(\alpha-1)\mu_s(x+k_{\alpha,\eps})^{
\frac{2(\alpha-1)}\alpha}\geq 2(\alpha-1)\eps \left(k_{\alpha,\eps}\right)^{
\frac{2(\alpha-1)}\alpha}\geq \frac{(1+\eps)^{\frac{2-\alpha}\alpha}}{(1+\eps)^{\frac{2-\alpha}\alpha}-1}
\end{equation}
and
\begin{equation}\label{eq:2.9}
\phi(s,x;\eps)\frac{2(\alpha-1)^2\mu_s}{\alpha^2(x+k_{\alpha,\eps})^{
\frac{2}\alpha}}\geq \Dis \exp\left(\eps \left(x+k_{\alpha,\eps}\right)^{
\frac{2(\alpha-1)}\alpha}\right)\frac{2(\alpha-1)^2\eps}{\alpha^2(x+k_{\alpha,\eps})^{
\frac{2}\alpha}}.\vspace{0.2cm}
\end{equation}
Observe that the function in the right hand side of \eqref{eq:2.9} is positive and continuous in $\R_+$, and tends to infinity as $x\To +\infty$. It follows that there exists a constant $\delta_{\alpha,\eps}>0$ depending only on $(\alpha,\eps)$ such that
\begin{equation}\label{eq:2.10}
\phi(s,x;\eps)\frac{2(\alpha-1)^2\mu_s}{\alpha^2(x+k_{\alpha,\eps})^{
\frac{2}\alpha}}\geq \delta_{\alpha,\eps},\ \ \ (s,x)\in \T\times \R_+.\vspace{0.2cm}
\end{equation}
Combining \eqref{eq:2.6}, \eqref{eq:2.8} and \eqref{eq:2.10} yields that for each $(s,x)\in \T\times \R_+$,\vspace{0.2cm}
\begin{equation}\label{eq:2.11}
\begin{array}{lll}
\Dis \phi_{xx}(s,x;\eps)-\delta_{\alpha,\eps}&\geq &\Dis \phi(s,x;\eps)\frac{2(\alpha-1)\mu_s\left[2(\alpha-1)\mu_s(x+k_{\alpha,\eps})^{
\frac{2(\alpha-1)}\alpha}\right]}{(1+\eps)^{\frac{2-\alpha}\alpha}\alpha^2
(x+k_{\alpha,\eps})^{
\frac{2}\alpha}}\vspace{0.2cm}\\
&=& \Dis \phi(s,x;\eps)\frac{4
(\alpha-1)^2\mu_s^2}{(1+\eps)^{\frac{2-\alpha}\alpha}\alpha^2
(x+k_{\alpha,\eps})^{\frac{2(2-\alpha)}\alpha}}.\vspace{0.2cm}
\end{array}
\end{equation}

In the sequel, we substitute \eqref{eq:2.5}, \eqref{eq:2.7} and \eqref{eq:2.11} into the left side of \eqref{eq:2.2} with $\delta=\delta_{\alpha,\eps}$ to obtain that for each $(s,x)\in \T\times\R_+$,
$$
\begin{array}{ll}
&\Dis -\beta\phi_x(s,x;\eps) x -c_{\alpha,\gamma}{\left(\phi_x(s,x;\eps)\right)^{\frac{2}{2-\alpha}}\over \left(\phi_{xx}(s,x;\eps)-\delta_{\alpha,\eps}\right)^{\frac\alpha{2-\alpha}}}
+\phi_s(s,x)\vspace{0.2cm}\\
\geq & \Dis -\beta\frac{2(\alpha-1)\mu_s\phi(s,x;\eps)(x+k_{\alpha,\eps})}{\alpha
(x+k_{\alpha,\eps})^{\frac{2-\alpha}
\alpha}}-c_{\alpha,\gamma}\frac{\left(\frac{2(\alpha-1)\mu_s}\alpha\phi(s,x;\eps)
\right)^{\frac{2}{2-\alpha}}(x+k_{\alpha,\eps})^{\frac{2(\alpha-1)}\alpha}}
{\left(\frac{4(\alpha-1)^2\mu_s^2}
{(1+\eps)^{\frac{2-\alpha}\alpha}\alpha^2}
\phi(s,x;\eps)\right)^{\frac\alpha{2-\alpha}}}\vspace{0.1cm}\\
& \Dis +\phi(s,x;\eps)
(x+k_{\alpha,\eps})^{\frac{2(\alpha-1)}
\alpha}\mu'_s\vspace{0.2cm}\\
=& \Dis  \phi(s,x;\eps)
(x+k_{\alpha,\eps})^{\frac{2(\alpha-1)}
\alpha}\left(-\frac{2(\alpha-1)\beta}\alpha\mu_s-c_{\alpha,\gamma}
\frac{1+\eps}{\left(\frac{2(\alpha-1)}\alpha\mu_s \right)^{\frac{2(\alpha-1)}{2-\alpha}}}+\mu'_s\right).\vspace{0.1cm}
\end{array}
$$
Thus, \eqref{eq:2.2} holds if the function $\mu_s, s\in \T$ satisfies the following ODE:
\begin{equation}\label{eq:2.12}
\mu'_s=\frac{2(\alpha-1)\beta}\alpha\mu_s+C_{\alpha,\gamma}\frac{1+\eps}
{\mu_s^{\frac{2(\alpha-1)}{2-\alpha}}},\ \ s\in \T
\end{equation}
with $\mu_0=\eps$ and
$$
C_{\alpha,\gamma}:=\frac{c_{\alpha,\gamma}}{\left(\frac{2(\alpha-1)}\alpha \right)^{\frac{2(\alpha-1)}{2-\alpha}}}=\frac{(2-\alpha)
\alpha^{\frac\alpha{2-\alpha}}
}{2\left(\frac{2(\alpha-1)}\alpha\right)^{\frac{2(\alpha-1)}{2-\alpha}}}
\gamma^{\frac{2}{2-\alpha}}.\vspace{0.2cm}
$$

It remains to solve ODE \eqref{eq:2.12}. We need to distinguish two different cases $\beta=0$ and $\beta>0$. We first consider the case of $\beta=0$. After separating variables for \eqref{eq:2.12} we integrate on the time interval $[0,s]$ to get that
$$
\frac{2-\alpha}\alpha\left(\mu_s^{\frac\alpha{2-\alpha}}-
\eps^{\frac\alpha{2-\alpha}}\right)=\left.\frac{2-\alpha}\alpha
\mu_r^{\frac\alpha{2-\alpha}}\right|_0^s =\int_0^s \mu_r^{\frac{2(\alpha-1)}{2-\alpha}} {\rm d}\mu_r =C_{\alpha,\gamma}(1+\eps)s
$$
and then
$$
\mu_s=\left(\frac\alpha{2-\alpha}C_{\alpha,\gamma}(1+\eps)s+
\eps^{\frac\alpha{2-\alpha}}
\right)^{\frac{2-\alpha}\alpha}, \ \ s\in \T.\vspace{0.2cm}
$$
For the case of $\beta>0$, after separating variables for \eqref{eq:2.12} again we integrate on interval $[0,s]$ to get that\vspace{0.1cm}
$$
\left.\frac{2-\alpha}{2(\alpha-1)\beta}\ln\left(\frac{2(\alpha-1)\beta}\alpha
\mu_r^{\frac\alpha{2-\alpha}}+C_{\alpha,\gamma}(1+\eps)\right)\right|_0^s   =\int_0^s \frac{\mu_r^{\frac{2(\alpha-1)}{2-\alpha}}{\rm d}\mu_r}{\frac{2(\alpha-1)\beta}\alpha
\mu_r^{\frac\alpha{2-\alpha}}+C_{\alpha,\gamma}(1+\eps)}=s
$$
and then
$$
\mu_s=\left\{\left(
\eps^{\frac\alpha{2-\alpha}}+\frac{C_{\alpha,\gamma}\alpha
(1+\eps)}{2(\alpha-1)\beta}\right)
\exp\left(\frac{2(\alpha-1)\beta}{2-\alpha}s\right)-\frac{C_{\alpha,\gamma}\alpha
(1+\eps)}{2(\alpha-1)\beta}\right\}^{\frac{2-\alpha}\alpha},\ \ s\in\T.\vspace{0.3cm}
$$

We summarize the preceding arguments into the following proposition, which will play a crucial role in the proof of the main results of this paper later.

\begin{pro}\label{Pro:2.1}
Given $\alpha\in (1,2)$ and $\beta,\gamma>0$. For each $(s,x)\in \T\times\R_+$ and $\eps>0$, define
\begin{equation}\label{eq:2.13}
\tilde \varphi(s,x;\eps):=\exp\left(\tilde\mu_{\alpha,\gamma,\eps}(s) \left(x+k_{\alpha,\eps}\right)^{\frac{2}{\alpha^*}}\right)=\exp\left(\tilde\mu_{\alpha,\gamma,\eps}(s) \left(x+k_{\alpha,\eps}\right)^{\frac{2(\alpha-1)}{\alpha}}\right)
\end{equation}
and
\begin{equation}\label{eq:2.14}
\bar \varphi(s,x;\eps):=\exp\left(\bar\mu_{\alpha,\beta,\gamma,\eps}(s) \left(x+k_{\alpha,\eps}\right)^{\frac{2}{\alpha^*}}\right)=\exp\left(\bar\mu_{\alpha,\beta,\gamma,\eps}(s) \left(x+k_{\alpha,\eps}\right)^{\frac{2(\alpha-1)}{\alpha}}\right),
\end{equation}
where
\begin{equation}\label{eq:2.15}
\tilde \mu_{\alpha,\gamma,\eps}(s):=\left(\tilde c_{\alpha,\gamma}(1+\eps)s+\eps^{\frac\alpha{2-\alpha}}
\right)^{\frac{2-\alpha}\alpha},\ \ \ k_{\alpha,\eps}:=\left(\frac{(1+\eps)^{\frac{2-\alpha}\alpha}}{2(\alpha-1) \eps \left((1+\eps)^{\frac{2-\alpha}\alpha}-1 \right)} \right)^{\frac\alpha{2(\alpha-1)}}
\end{equation}
and
\begin{equation}\label{eq:2.16}
\bar \mu_{\alpha,\beta,\gamma,\eps}(s):=\left\{\left(
\eps^{\frac\alpha{2-\alpha}}+(1+\eps)\bar c_{\alpha,\beta,\gamma}\right)
\exp\left(\frac{2(\alpha-1)\beta}{2-\alpha}s\right)-(1+\eps)\bar c_{\alpha,\beta,\gamma}\right\}^{\frac{2-\alpha}\alpha}
\end{equation}
with
\begin{equation}\label{eq:2.17}
\tilde c_{\alpha,\gamma}:=\frac{
(\alpha\gamma)^{\frac{2}{2-\alpha}}
}{2\left(\frac{2(\alpha-1)}\alpha\right)^{\frac{2(\alpha-1)}{2-\alpha}}}\ \ \ {\rm and}\ \ \
\bar c_{\alpha,\beta,\gamma}:=\frac{(2-\alpha)(\alpha\gamma)^{\frac{2}{2-\alpha}}
}{4\alpha\beta\left(\frac{2(\alpha-1)}\alpha\right)^{\frac{2(\alpha-1)}{2-\alpha}}}.\vspace{0.2cm}
\end{equation}
Then, there exists a constant $\delta_{\alpha,\eps}>0$ depending only on $(\alpha,\eps)$ such that for each $(s,x,z)\in \T\times \R_+\times\R^d$, we have
\begin{equation}\label{eq:2.18}
-\gamma \tilde\varphi_x(s,x;\eps)|z|^\alpha+{1\over 2}(\tilde\varphi_{xx}(s,x;\eps)-\delta_{\alpha,\eps})|z|^2+\tilde\varphi_s(s,x;\eps)\geq 0
\end{equation}
and
\begin{equation}\label{eq:2.19}
-\bar\varphi_x(s,x;\eps) \left(\beta x+\gamma |z|^\alpha\right) +{1\over 2}(\bar\varphi_{xx}(s,x;\eps)-\delta_{\alpha,\eps})|z|^2+\bar\varphi_s(s,x;\eps)\geq 0.\vspace{0.3cm}
\end{equation}
\end{pro}

Finally, for the uniqueness, comparison theorem and stability theorem of unbounded solutions to BSDE$(\xi,g)$ with generator $g$ satisfying \eqref{eq:1.2} with $\alpha\in (1,2)$, some stronger assumptions than those needed for the existence are required as usual. We first assume in addition that $(\xi,g_0)$ has sub-exponential moments integrability of any order and the generator $g$ is convex or concave with respect to the state variables $(y,z)$, which appears a natural assumption for a non-linear growth function (see e.g. \citet{BriandHu2008PTRF} and \citet{DelbaenHuRichou2011AIHPPS}), and then relax the convexity (concavity) assumption. The main idea is to use the $\theta$-technique developed in \citet{BriandHu2008PTRF} to prove these results. More specifically, in order to take
advantage of the (extended) convexity condition, we will estimate $Y^1_\cdot-\theta Y^2_\cdot$, for each $\theta\in (0,1)$, instead of estimating the difference between the processes $Y^1_\cdot$ and $Y^2_\cdot$. Moreover, it turns out that the uniform a priori estimate is also the key to solve the uniqueness, comparison theorem and stability theorem of the solutions.\vspace{-0.1cm}

\section{Existence of the solution}
\label{sec:3-Main results}
\setcounter{equation}{0}

In this section, we assume that $\xi$ is a terminal condition and $g$ is a generator which is continuous in $(y,z)$, and satisfies the following assumption:
\begin{enumerate}
\renewcommand{\theenumi}{(H1)}
\renewcommand{\labelenumi}{\theenumi}
\item\label{H1} There exist three constants $\alpha\in (1,2)$, $\beta\geq 0$ and $\gamma>0$ such that $\as$,
    $$
    |g(\omega,t,y,z)|\leq |g(\omega,t,0,0)|+\beta|y|+\gamma |z|^\alpha,\ \ \ (y,z)\in \R\times\R^d.
    $$
\end{enumerate}

Define the function\vspace{-0.1cm}
\begin{equation}\label{eq:3.1}
\psi(x,\mu):=\exp\left(\mu~x^{\frac{2}{\alpha^*}}\right)=\exp\left(\mu~x^{\frac{2(\alpha-1)}{\alpha}}\right),\ \ (x,\mu)\in \R_+\times \R_+,\vspace{-0.1cm}
\end{equation}
and the two constants
\begin{equation}\label{eq:3.2}
\begin{array}{l}
\Dis \tilde\mu^0_{\alpha,\gamma,T}:=(\tilde c_{\alpha,\gamma}T)^{\frac{2-\alpha}\alpha}=
\frac{\alpha^2}
{2(\alpha-1)^{\frac{2(\alpha-1)}\alpha}}\gamma^{\frac{2}\alpha}
T^{\frac{2-\alpha}\alpha},\vspace{0.2cm}\\
\Dis \bar\mu^0_{\alpha,\beta,\gamma,T}:=\left\{\bar c_{\alpha,\beta,\gamma}\exp\left(\frac{2(\alpha-1)\beta}{2-\alpha}T \right)-\bar c_{\alpha,\beta,\gamma}\right\}^{\frac{2-\alpha}\alpha},\vspace{0.1cm}
\end{array}
\end{equation}
where $\tilde c_{\alpha,\gamma}$ and $\bar c_{\alpha,\beta,\gamma}$ are defined in \eqref{eq:2.17}. Note that
$\tilde\mu^0_{\alpha,\gamma,T}=\lim_{\eps\To 0^+} \tilde\mu_{\alpha,\gamma,\eps}(T)$ and $\bar\mu^0_{\alpha,\beta,\gamma,T}=\lim_{\eps\To 0^+} \bar\mu_{\alpha,\beta,\gamma,\eps}(T)$,
where $\tilde\mu_{\alpha,\gamma,\eps}(\cdot)$ and $\bar\mu_{\alpha,\beta,\gamma,\eps}(\cdot)$ are defined in \eqref{eq:2.15} and \eqref{eq:2.16} \vspace{0.1cm}respectively.

The following existence theorem is the main result of this section.\vspace{-0.1cm}

\begin{thm}\label{thm:3.1}
Assume that $\xi$ is a terminal condition, $g$ is a generator which is continuous with respect to $(y,z)$ and satisfies assumption \ref{H1} with parameters $\alpha$, $\beta$ and $\gamma$, and the function $\psi(x,\mu)$ and the constants $\tilde\mu^0_{\alpha,\gamma,T}$ and $\bar\mu^0_{\alpha,\beta,\gamma,T}$ are defined respectively in \eqref{eq:3.1} and \eqref{eq:3.2}. \vspace{0.1cm}

(i) Let $\beta=0$ and $\tilde \mu_{\alpha,\gamma,\eps}(\cdot)$ be defined in \eqref{eq:2.15}. If there exists a constant $\mu>\tilde\mu^0_{\alpha,\gamma,T}$ such that
\begin{equation}\label{eq:3.3}
\E\left[\psi\left(|\xi|+\int_0^T|g(t,0,0)|{\rm d}t,\ \mu\right)\right]=\E\left[\exp\left\{\mu \left(|\xi|+\int_0^T|g(t,0,0)|{\rm d}t\right)^{\frac{2}{\alpha^*}} \right\}\right]<+\infty,\vspace{0.1cm}
\end{equation}
then BSDE$(\xi,g)$ admits a solution $(Y_t,Z_t)_{t\in\T}$ such that $\left(\psi\left(|Y_t|,\ \tilde\mu_{\alpha,\gamma,\eps}(t)\right)\right)_{t\in\T}$ belongs to class (D) for some $\eps>0$ and $Z_\cdot\in \mcal^2$. Moreover, for some constant $\delta_{\alpha,\eps}>0$ depending only on $(\alpha,\eps)$ we have, $\ps$,
\begin{equation}\label{eq:3.4}
\begin{array}{lll}
&&\Dis \psi\left(|Y_t|,\ \tilde\mu_{\alpha,\gamma,\eps}(t)\right)+\frac{\delta_{\alpha,\eps}}{2}
\E\left[\left.\int_t^T|Z_s|^2{\rm d}s\right|\F_t\right]\vspace{0.2cm} \\
&\leq & \Dis C_{\mu,\alpha,\eps}\E\left[\left.\psi\left(|\xi|+\int_0^T|g(s,0,0)|{\rm d}s,\  \tilde\mu_{\alpha,\gamma,\eps}(T)\right)\right|\F_t\right]\vspace{0.2cm}\\
&= & \Dis C_{\mu,\alpha,\eps}\E\left[\left.\psi\left(|\xi|+\int_0^T|g(s,0,0)|{\rm d}s, \ \mu\right)\right|\F_t\right],\ \ \ \ \ t\in \T,
\end{array}
\end{equation}
where $C_{\mu,\alpha,\eps}$ is a positive constant depending only on $(\mu,\alpha,\eps)$.\vspace{0.2cm}

(ii) Let $\beta>0$ and $\bar\mu_{\alpha,\beta,\gamma,\eps}(\cdot)$ be defined in \eqref{eq:2.16}. If there exists a constant $\mu>\bar\mu^0_{\alpha,\beta,\gamma,T}$ such that \eqref{eq:3.3} holds, then BSDE$(\xi,g)$ admits a solution $(Y_t,Z_t)_{t\in\T}$ such that $\left(\psi\left(|Y_t|,\ \bar\mu_{\alpha,\beta,\gamma,\eps}(t)\right)\right)_{t\in\T}$ belongs to class (D) for some $\eps>0$ and $Z_\cdot\in \mcal^2$. Moreover, for some $\delta_{\alpha,\eps}>0$ depending only on $(\alpha,\eps)$ we have, $\ps$,
\begin{equation}\label{eq:3.5}
\begin{array}{lll}
&&\Dis\psi\left(|Y_t|,\ \bar\mu_{\alpha,\beta,\gamma,\eps}(t)\right)+\frac{\delta_{\alpha,\eps}}{2}
\E\left[\left.\int_t^T|Z_s|^2{\rm d}s\right|\F_t\right]\vspace{0.2cm} \\
&\leq & \Dis C_{\mu,\alpha,\eps}\E\left[\left.\psi\left(|\xi|+\int_0^T|g(s,0,0)|{\rm d}s, \ \bar\mu_{\alpha,\beta,\gamma,\eps}(T)\right)\right|\F_t\right]\vspace{0.2cm}\\
&= & \Dis C_{\mu,\alpha,\eps}\E\left[\left.\psi\left(|\xi|+\int_0^T|g(s,0,0)|{\rm d}s, \ \mu\right)\right|\F_t\right],\ \ \ t\in\T,
\end{array}
\end{equation}
where the constant $C_{\mu,\alpha,\eps}$ is the same as (i).
\end{thm}

\begin{rmk}\label{rmk:3.2}
It is not very hard to check that $\tilde\mu^0_{\alpha,\gamma,T}$ and $\bar\mu^0_{\alpha,\beta,\gamma,T}$ defined in \eqref{eq:3.2} tends respectively to $2\gamma $ and $2\gamma e^{\beta T}$ as $\alpha \To 2$, which is a direct correspondence of the known result for the quadratic growth case in \citet{BriandHu2006PTRF,BriandHu2008PTRF}. From
this point of view, the condition \eqref{eq:3.3} in \cref{thm:3.1} is seemingly the reasonably weakest possible one guaranteeing the existence of the solution. However, by now we can not prove it.
\end{rmk}

In order to prove \cref{thm:3.1}, we need the following proposition, which establishes some a priori estimate for solutions to BSDEs with bounded terminal conditions and sub-quadratic growth generators.

\begin{pro}\label{Pro:3.3}
Assume that $\xi$ is a terminal condition, $g$ is a generator which is continuous in the state variables $(y,z)$ and satisfies assumption \ref{H1} with parameters $\alpha$, $\beta$ and $\gamma$, and the functions $\tilde \mu_{\alpha,\gamma,\eps}(s)$, $\bar \mu_{\alpha,\beta,\gamma,\eps}(s)$ and $\psi(x,\mu)$ together with the constant $k_{\alpha,\eps}$ are respectively defined in \eqref{eq:2.15}, \eqref{eq:2.16} and \eqref{eq:3.1}.\vspace{0.2cm}

Let $|\xi|+\int_0^T|g(t,0,0)|{\rm d}t$ be a bounded random variable, and $(Y_t,Z_t)_{t\in\T}$ a solution to BSDE$(\xi,g)$ such that $Y_\cdot$ is a bounded process (and $Z_\cdot\in \mcal^2$). Then for any $\eps>0$, there exists a constant $\delta_{\alpha,\eps}>0$ depending only on $(\alpha,\eps)$ such that $\ps$, for each $t\in\T$, the inequality
\begin{equation}\label{eq:3.6}
\begin{array}{lll}
&&\Dis\psi\left(|Y_t|,\ \tilde\mu_{\alpha,\gamma,\eps}(t)\right)+\frac{\delta_{\alpha,\eps}}{2}
\E\left[\left.\int_t^T|Z_s|^2{\rm d}s\right|\F_t\right]\vspace{0.2cm} \\
&\leq & \Dis \exp\left(\tilde\mu_{\alpha,\gamma,\eps}(T) \ k_{\alpha,\eps}^{\frac{2}{\alpha^*}}\right) \E\left[\left.\psi\left(|\xi|+\int_0^T|g(s,0,0)|{\rm d}s,\  \tilde\mu_{\alpha,\gamma,\eps}(T)\right)\right|\F_t\right]\vspace{0.2cm}
\end{array}
\end{equation}
holds for $\beta=0$, and the inequality
\begin{equation}\label{eq:3.7}
\begin{array}{lll}
&&\Dis
\psi\left(|Y_t|,\ \bar\mu_{\alpha,\beta,\gamma,\eps}(t)\right)+\frac{\delta_{\alpha,\eps}}{2}
\E\left[\left.\int_t^T|Z_s|^2{\rm d}s\right|\F_t\right]\vspace{0.2cm} \\
&\leq & \Dis \exp\left(\bar\mu_{\alpha,\beta,\gamma,\eps}(T) \ k_{\alpha,\eps}^{\frac{2}{\alpha^*}}\right) \E\left[\left.\psi\left(|\xi|+\int_0^T|g(s,0,0)|{\rm d}s, \ \bar\mu_{\alpha,\beta,\gamma,\eps}(T)\right)\right|\F_t\right]
\end{array}
\end{equation}
holds for $\beta>0$.
\end{pro}

\begin{proof}
We first consider the case of $\beta=0$. Define
$$
\bar Y_t:=|Y_t|+\int_0^t |g(s,0,0)|{\rm d}s\ \ \ \
{\rm and}\ \ \ \ \bar Z_t:={\rm sgn}(Y_t)Z_t,\ \ \ \ t\in \T.
$$
It follows from It\^{o}-Tanaka's formula that
$$
\bar Y_t=\bar Y_T+\int_t^T \left({\rm sgn}(Y_s)g(s,Y_s,Z_s)-|g(s,0,0)|\right){\rm d}s-\int_t^T \bar Z_s \cdot {\rm d}B_s-\int_t^T {\rm d}L_s, \ \ \ t\in\T,
$$
where $L_\cdot$ stands for the local time of $Y_\cdot$ at $0$. Now, we fix $\eps>0$ and apply It\^{o}-Tanaka's formula to the process $\tilde\varphi(s, \bar Y_s; \eps)$, where the function $\tilde\varphi(s,x;\eps)$ is defined in \eqref{eq:2.13}, to derive, in view of assumption \ref{H1} with $\beta=0$,
$$
\begin{array}{lll}
\Dis {\rm d}\tilde\varphi(s,\bar Y_s;\eps)&=&\Dis \tilde\varphi_x(s,\bar Y_s;\eps)
\left(-{\rm sgn}(Y_s)g(s,Y_s,Z_s)+|g(s,0,0)|\right){\rm d}s+\tilde\varphi_x(s,\bar Y_s;\eps)\bar Z_s \cdot {\rm d}B_s\vspace{0.1cm}\\
&&\Dis +\tilde\varphi_x(s,\bar Y_s;\eps){\rm d}L_s+{1\over 2}\tilde\varphi_{xx}(s,\bar Y_s;\eps)|Z_s|^2{\rm d}s+\tilde\varphi_s(s,\bar Y_s;\eps){\rm d}s
\vspace{0.1cm}\\
&\geq &\left[-\gamma \tilde\varphi_x(s,\bar Y_s;\eps)|Z_s|^\alpha+{1\over 2}\tilde\varphi_{xx}(s,\bar Y_s;\eps)|Z_s|^2+\tilde\varphi_s(s,\bar Y_s;\eps)\right]{\rm d}s\vspace{0.1cm}\\
&&\Dis +\tilde\varphi_x(s,\bar Y_s;\eps)\bar Z_s \cdot {\rm d}B_s.
\end{array}
$$
Thus, from \eqref{eq:2.18} in \cref{Pro:2.1} we know the existence of a positive constant $\delta_{\alpha,\eps}>0$ depending only on $\alpha,\eps$ such that
\begin{equation}\label{eq:3.8}
{\rm d}\tilde\varphi(s,\bar Y_s;\eps)\geq \frac{1}{2}\delta_{\alpha,\eps}|Z_s|^2 {\rm d}s+\tilde\varphi_x(s,\bar Y_s;\eps)\bar Z_s \cdot {\rm d}B_s,\ \ s\in \T.
\end{equation}
Let us denote, for each $t\in\T$ and each integer $m\geq 1$, the following stopping time
$$
\tau_m:=\inf\left\{s\in [t,T]: \int_t^s \left(\tilde\varphi_x(r,\bar Y_r;\eps)\right)^2|\bar Z_r|^2{\rm d}r\geq m \right\}\wedge T
$$
with the convention $\inf\emptyset=+\infty$. It follows from the inequality \eqref{eq:3.8} and the definition of $\tau_m$ that for each $t\in \T$ and $m\geq 1$,
$$
\tilde \varphi(t,\bar Y_t;\eps)+ \frac{\delta_{\alpha,\eps}}{2}\E\left[\left.\int_t^{\tau_m}|Z_s|^2 {\rm d}s\right|\F_t\right] \leq \E\left[\left. \tilde\varphi(\tau_m,\bar Y_{\tau_m};\eps)\right|\F_t\right].
$$
Furthermore, in view of the definition of $\tau_m$ again, by sending $m$ to infinity and using Fatou's lemma and Lebesgue's dominated convergence theorem in above inequality we get
\begin{equation}\label{eq:3.9}
\tilde \varphi(t,\bar Y_t;\eps)+ \frac{\delta_{\alpha,\eps}}{2}\E\left[\left.\int_t^T |Z_s|^2 {\rm d}s\right|\F_t\right] \leq \E\left[\left. \tilde\varphi(T,\bar Y_T;\eps)\right|\F_t\right],\ \ \ t\in\T.
\end{equation}
And, from the definitions of $\tilde\varphi(s,x;\eps)$ and $\psi(x,\mu)$ with the inequality $(a+b)^\lambda\leq a^\lambda+b^\lambda$ for $a,b\geq 0$ and $\lambda\in (0,1)$, observe that for each $x\in\R_+$, $t\in\T$ and $\eps>0$,
\begin{equation}\label{eq:3.10}
\psi(x,\tilde\mu_{\alpha,\gamma,\eps}(t)) \leq \tilde\varphi(t,x;\eps)\leq \exp\left(\tilde\mu_{\alpha,\gamma,\eps}(t) \ k_{\alpha,\eps}^{\frac{2}{\alpha^*}}\right) \psi(x,\tilde\mu_{\alpha,\gamma,\eps}(t)).
\end{equation}
The desired inequality \eqref{eq:3.6} follows immediately from \eqref{eq:3.9} and \eqref{eq:3.10}.\vspace{0.1cm}

Finally, in the case of $\beta>0$, by a similar argument as above we can use the functions $\bar\varphi(s,x;\eps)$ and $\bar\mu_{\alpha,\beta,\gamma,\eps}(t)$ defined respectively in \eqref{eq:2.14} and \eqref{eq:2.16} of \cref{Pro:2.1} instead of $\tilde\varphi(s,x;\eps)$ and $\tilde\mu_{\alpha,\gamma,\eps}(t)$, and apply \eqref{eq:2.19} in \cref{Pro:2.1} to get the desired inequality \eqref{eq:3.7}. The proof is then completed.\vspace{0.1cm}
\end{proof}

\begin{rmk}\label{rmk:3.4}
From the above proof, it is easy to see that in \cref{Pro:3.3}, if   $|Y_\cdot|$ and $|\xi|$ are replaced with $Y^+_\cdot$ and $\xi^+$ respectively, and \ref{H1} is replaced with the following assumption \ref{H1'}:
\begin{enumerate}
\renewcommand{\theenumi}{(H1')}
\renewcommand{\labelenumi}{\theenumi}
\item\label{H1'} There exist three constants $\alpha\in (1,2)$, $\beta\geq 0$ and $\gamma>0$ such that $\as$,
  \[
    g(\omega,t,y,z){\bf 1}_{y>0} \leq |g(\omega,t,0,0)|+\beta|y|+\gamma |z|^\alpha,\ \ \ (y,z)\in \R\times\R^d,
  \]
\end{enumerate}
then the conclusions of \cref{Pro:3.3} still hold for $Y^+_\cdot$ and $\xi^+$, but the term $|Z_s|^2$ in \eqref{eq:3.6} and \eqref{eq:3.7} needs to be replaced with ${\bf 1}_{Y_s>0}|Z_s|^2$. For this, in the above proof one needs to respectively use $Y_\cdot^+$, ${\bf 1}_{Y_\cdot>0} Y_\cdot$ and ${1\over 2}L_\cdot$ instead of $|Y_\cdot|$, ${\rm sgn} (Y_\cdot)$ and $L_\cdot$.\vspace{0.1cm}
\end{rmk}

Now, we can give the proof of \cref{thm:3.1}.

\begin{proof}[The proof of \cref{thm:3.1}]
For any given positive integers $n,p\geq 1$, set
$$
\xi^{n,p}:=\xi^+\wedge n-\xi^-\wedge p\ \ \ \ {\rm and}\ \ \ \ g^{n,p}(\omega,t,y,z):=g^+(\omega,t,y,z)\wedge n-g^-(\omega,t,y,z)\wedge p.
$$
As both the terminal condition $\xi^{n,p}$ and the generator $g^{n,p}$ are bounded and $g^{n,p}(t,y,z)$ remains to be continuous in $(y,z)$, in view of the existence result in \citet{LepeltierSanMartin1997SPL}, the following BSDE$(\xi^{n,p},g^{n,p})$ admits a maximal bounded solution  $(Y^{n,p}_t,Z^{n,p}_t)_{t\in\T}$ such that $Y^{n,p}_\cdot$ is a bounded process and $Z^{n,p}_\cdot\in \mcal^2$:
\begin{equation}\label{eq:3.11}
  Y^{n,p}_t=\xi^{n,p}+\int_t^T g^{n,p}(s,Y^{n,p}_s,Z^{n,p}_s){\rm d}s-\int_t^T Z^{n,p}_s \cdot {\rm d}B_s, \ \ t\in\T.
\end{equation}
And, by virtue of the comparison theorem, $Y^{n,p}_\cdot$ is nondecreasing in $n$ and non-increasing in $p$.\vspace{0.2cm}

We now assume that $\beta=0$ and there exists a constant $\mu>\tilde\mu^0_{\alpha,\gamma,T}$ such that \eqref{eq:3.3} holds. Observe that the function\vspace{-0.1cm}
$$
\tilde \mu_{\alpha,\gamma,\eps}(t):=\left(\tilde c_{\alpha,\gamma}(1+\eps)t+\eps^{\frac\alpha{2-\alpha}}
\right)^{\frac{2-\alpha}\alpha}
$$
defined in \eqref{eq:2.15} is strictly increasing with respect to the variables $t\in\T$ and $\eps>0$, $\tilde \mu_{\alpha,\gamma,\eps}(T)\To\tilde\mu^0_{\alpha,\gamma,T}$ when $\eps\To 0^+$ and $\tilde \mu_{\alpha,\gamma,\eps}(T)\To +\infty$ when $\eps\To +\infty$. Since $\mu>\tilde\mu^0_{\alpha,\gamma,T}$, we can conclude that there must exist a positive $\eps_0>0$ such that for each $t\in\T$,
\begin{equation}\label{eq:3.12}
\eps_0=\tilde \mu_{\alpha,\gamma,\eps_0}(0)\leq \tilde \mu_{\alpha,\gamma,\eps_0}(t)\leq \tilde \mu_{\alpha,\gamma,\eps_0}(T)=\mu.
\end{equation}
Thus, we can apply \eqref{eq:3.6} in \cref{Pro:3.3} with $\eps=\eps_0$ for BSDE \eqref{eq:3.11} to get that there exists a positive constant $\delta_{\alpha,\eps_0}>0$ depending only on $(\alpha,\eps_0)$ such that $\ps$, for each $t\in\T$ and $n,p\geq 1$,
\begin{equation}\label{eq:3.13}
\begin{array}{ll}
&\Dis\psi\left(|Y_t^{n,p}|,\ \eps_0\right)\vspace{0.1cm}\\
\leq &\Dis\psi\left(|Y_t^{n,p}|,\ \tilde\mu_{\alpha,\gamma,\eps_0}(t)\right)+\frac{\delta_{\alpha,\eps_0}}{2}
\E\left[\left.\int_t^T|Z_s^{n,p}|^2{\rm d}s\right|\F_t\right]\vspace{0.2cm} \\
\leq & \Dis \exp\left(\tilde\mu_{\alpha,\gamma,\eps_0}(T) \ k_{\alpha,\eps_0}^{\frac{2}{\alpha^*}}\right) \E\left[\left.\psi\left(|\xi^{n,p}|+\int_0^T|g^{n,p}(s,0,0)|{\rm d}s,\  \tilde\mu_{\alpha,\gamma,\eps_0}(T)\right)\right|\F_t\right]\vspace{0.2cm} \\
\leq & \Dis \exp\left(\mu\ k_{\alpha,\eps_0}^{\frac{2}{\alpha^*}}\right) \E\left[\left.\psi\left(|\xi|+\int_0^T|g(s,0,0)|{\rm d}s,\  \tilde\mu_{\alpha,\gamma,\eps_0}(T)\right)\right|\F_t\right]\vspace{0.2cm} \\
= & \Dis \exp\left(\mu\ k_{\alpha,\eps_0}^{\frac{2}{\alpha^*}}\right) \E\left[\left.\psi\left(|\xi|+\int_0^T|g(s,0,0)|{\rm d}s,\  \mu\right)\right|\F_t\right]<+\infty.\vspace{0.1cm}
\end{array}
\end{equation}
In previous inequality, we have used \eqref{eq:3.12} together with definitions of $\xi^{n,p}$ and $g^{n,p}$. Now, in view of assumption \ref{H1} and the fact that, by \eqref{eq:3.13},
$$
\begin{array}{lll}
\Dis |Y_t^{n,p}|&=& \Dis \left(\frac{1}{\eps_0}\ln\left(\psi\left(|Y_t^{n,p}|, \eps_0\right)\right)\right)^{\frac{\alpha^*}{2}}\vspace{0.2cm}\\
&\leq &\Dis \left(\frac{1}{\eps_0}\mu k_{\alpha,\eps_0}^{\frac{2}{\alpha^*}}+\frac{1}{\eps_0}\ln\left\{ \E\left[\left.\psi\left(|\xi|+\int_0^T|g(s,0,0)|{\rm d}s, \mu\right)\right|\F_t\right]\right\}\right)^{\frac{\alpha^*}{2}},
\vspace{0.2cm}
\end{array}
$$
we can apply the localization procedure developed initially in
\citet{BriandHu2006PTRF} to obtain the existence of a progressively measurable process $(Z_t)_{t\in\T} $ such that $\as$, $Z^{n,p}_\cdot$ tends to $Z_\cdot$ as $n, p$ tends to infinity and the pair of $(Y_\cdot:=\inf_p\sup_n Y^{n,p}_\cdot, \ Z_\cdot)$ is a solution to BSDE$(\xi,g)$. Moreover, we can send $n$ and $p$ to infinity in \eqref{eq:3.13} and use Fatou's lemma to get the inequality \eqref{eq:3.4}, and then $\left(\psi\left(|Y_t|, \tilde\mu_{\alpha,\gamma,\eps_0}(t)\right)\right)_{t\in\T}$ belongs to class (D), and $Z_\cdot\in\mcal^2$.\vspace{0.2cm}

Finally, in the case of $\beta>0$, by a similar argument as above we can use $\bar\mu_{\alpha,\beta,\gamma,\eps}(t)$ defined in \eqref{eq:2.16} of \cref{Pro:2.1} instead of $\tilde\mu_{\alpha,\gamma,\eps}(t) $, and apply \eqref{eq:3.7} with $\eps=\eps_0$ in \cref{Pro:3.3} instead of \eqref{eq:3.6} to get the desired inequality \eqref{eq:3.5}. The theorem is then proved.
\end{proof}

\begin{rmk}\label{rmk:3.5}
From the above proof, it is not very difficult to see that the sub-quadratic growth assumption \ref{H1} in \cref{thm:3.1} and \cref{Pro:3.3} can be relaxed to the following one-sided sub-quadratic growth assumption, which will be used in \cref{sec:5} and \cref{sec:6},
\begin{enumerate}
\renewcommand{\theenumi}{(H1")}
\renewcommand{\labelenumi}{\theenumi}
\item\label{H1"}
There exist four real constants $\alpha\in (1,2)$, $\beta\geq 0$, $\gamma>0$ and $c>0$, and a progressively measurable $\R_+$-valued process $(f_t)_{t\in \T}$ such that $\as$, for each $(y,z)\in \R\times\R^d$,\vspace{0.1cm}
$$
{\rm sgn}(y)g(\omega,t,y,z)\leq f_t(\omega)+\beta|y|+\gamma |z|^\alpha\ \ \ \ {\rm and}\ \ \ \ |g(\omega,t,y,z)|\leq f_t(\omega)+ h(|y|)+c|z|^2,
$$
where $h(\cdot)$ is a nondecreasing, continuous and deterministic
function with $h(0)=0$.
\end{enumerate}
In this case, one only needs to replace the process $|g(t,0,0)|$ in the conditions of \cref{thm:3.1} and \cref{Pro:3.3} with the process $f_t$.
\end{rmk}\vspace{-0.2cm}

\section{Uniqueness and comparison theorem of the solutions}
\label{sec:4}
\setcounter{equation}{0}

In this section, we will prove the uniqueness and comparison theorem for the unbounded solutions of BSDE \eqref{eq:1.1} with the terminal condition $\xi$ and the generator $g$ satisfying assumption \ref{H1} with parameters $\alpha,\beta$ and $\gamma$, and the following two assumptions \ref{H2} and \ref{H3}:
\begin{enumerate}
\renewcommand{\theenumi}{(H2)}
\renewcommand{\labelenumi}{\theenumi}
\item\label{H2}
$\as$, \ the generator $g$ is convex or concave with respect to the variables $(y,z)$.
\end{enumerate}

\begin{enumerate}
\renewcommand{\theenumi}{(H3)}
\renewcommand{\labelenumi}{\theenumi}
\item\label{H3}
The terminal condition $\xi+\int_0^T |g(t,0,0)|{\rm d}t$ has sub-exponential moments of any order, i.e., for any $p>0$, we have
\begin{equation}\label{eq:4.1}
\begin{array}{lll}
\hspace*{-0.5cm}\Dis \E\left[\psi\left(|\xi|+\int_0^T|g(t,0,0)|{\rm d}t,\ p\right)\right]& =& \Dis \E\left[\exp\left\{p \left(|\xi|+\int_0^T|g(t,0,0)|{\rm d}t\right)^{\frac{2}{\alpha^*}} \right\}\right]\vspace{0.1cm}\\
&<& +\infty.\vspace{0.1cm}
\end{array}
\end{equation}
\end{enumerate}

\begin{thm}\label{thm:4.1}
Assume that $\xi$ is a terminal condition, $g$ is a generator which is continuous in $(y,z)$ and satisfies assumption \ref{H1} with parameters $\alpha$, $\beta$ and $\gamma$.\vspace{0.1cm}

If the generator and the terminal condition further satisfy assumptions \ref{H2} and \ref{H3}, then BSDE$(\xi,g)$ admits a unique solution $(Y_t,Z_t)_{t\in\T}$ such that $\sup_{t\in\T}|Y_t|$ has sub-exponential moments of any order, i.e.,
\begin{equation}\label{eq:4.2}
\RE p>0,\ \ \ \E\left[\psi(\sup\limits_{t\in\T}|Y_t|,\ p)\right]=\E\left[\exp\left\{p \left(\sup\limits_{t\in\T}|Y_t|\right)^{\frac{2}{\alpha^*}} \right\}\right]<+\infty.
\end{equation}
Furthermore, $Z_\cdot\in \mcal^p$ for all $p>0$.
\end{thm}

\begin{proof} Firstly, since the generator $g$ satisfies \ref{H1} and \eqref{eq:4.1} holds, it follows from \cref{thm:3.1} together with its proof that BSDE$(\xi,g)$ admits a solution $(Y_t,Z_t)_{t\in\T}$ such that for each $\eps>0$, $\psi\left(|Y_t|,\ \tilde\mu_{\alpha,\gamma,\eps}(t)\right)$ belongs to class (D) for $\beta=0$, $\psi\left(|Y_t|,\ \bar\mu_{\alpha,\beta,\gamma,\eps}(t)\right)$ belongs to class (D) for $\beta>0$, and $Z_\cdot\in \mcal^2$.\vspace{0.1cm}

Now, we show \eqref{eq:4.2}. Indeed, since \eqref{eq:4.1} holds, by virtue of \cref{Pro:3.3} and \cref{thm:3.1} together with their proofs we can conclude that for each $p>0$, $\ps$, for each $t\in\T$,
$$
\psi\left(|Y_t|,\ p\right)\leq \psi\left(|Y_t|,\ \tilde\mu_{\alpha,\gamma,p}(t)\right)\leq \tilde C_{\alpha,\gamma,p,T}\E\left[\left.\psi\left(|\xi|+\int_0^T|g(s,0,0)|{\rm d}s,\  \tilde\mu_{\alpha,\gamma,p}(T)\right)\right|\F_t\right]
$$
holds for $\beta=0$, and
$$
\psi\left(|Y_t|,\ p\right)\leq
\psi\left(|Y_t|,\ \bar\mu_{\alpha,\beta,\gamma,p}(t)\right)\leq \bar C_{\alpha,\beta,\gamma,p,T}\E\left[\left.\psi\left(|\xi|+\int_0^T|g(s,0,0)|{\rm d}s, \ \bar\mu_{\alpha,\beta,\gamma,p}(T)\right)\right|\F_t\right]
$$
holds for $\beta>0$, where $\tilde\mu_{\alpha,\gamma,p}(\cdot)$, $k_{\alpha,p}$ and $\bar\mu_{\alpha,\beta,\gamma,p}(\cdot)$ are respectively defined in \eqref{eq:2.15} and \eqref{eq:2.16},
\begin{equation}\label{eq:4.3}
\tilde C_{\alpha,\gamma,p,T}:= \exp\left(\tilde\mu_{\alpha,\gamma,p}(T) \ k_{\alpha,p}^{\frac{2}{\alpha^*}}\right) \ \ \ {\rm and}\ \ \ \bar C_{\alpha,\beta,\gamma,p,T}:=\exp\left(\bar\mu_{\alpha,\beta,\gamma,p}(T) \ k_{\alpha,p}^{\frac{2}{\alpha^*}}\right).
\end{equation}
Consequently, in the case of $\beta=0$, for each $t\in\T$ and $p>0$, we can derive
\begin{equation}\label{eq:4.4}
\psi\left(\sup\limits_{t\in\T}|Y_t|,\ p\right)\leq \tilde C_{\alpha,\gamma,p,T}\sup\limits_{t\in\T}\left\{\E\left[\left.\psi
\left(|\xi|+\int_0^T|g(s,0,0)|{\rm d}s,\  \tilde\mu_{\alpha,\gamma,p}(T)\right)\right|\F_t\right]\right\},
\end{equation}
and in the case of $\beta>0$,
\begin{equation}\label{eq:4.5}
\psi\left(\sup\limits_{t\in\T}|Y_t|,\ p\right)\leq \bar C_{\alpha,\beta,\gamma,p,T}\sup\limits_{t\in\T}\left\{\E\left[\left.\psi
\left(|\xi|+\int_0^T|g(s,0,0)|{\rm d}s,\  \bar\mu_{\alpha,\beta,\gamma,p}(T)\right)\right|\F_t\right]\right\}.\vspace{0.2cm}
\end{equation}
Thus, with the help of Doob's maximal inequality on martingale, the desired inequality \eqref{eq:4.2} follows from inequalities \eqref{eq:4.4}, \eqref{eq:4.5} and \eqref{eq:4.1}.\vspace{0.1cm}

In the sequel, we prove that $Z_\cdot\in \mcal^p$ for all $p>0$. We only prove the case of $\beta=0$, and the case of $\beta>0$ can be proved in the same way. Let the function $\tilde\varphi(s,x;\eps)$ be defined in \eqref{eq:2.13}. In the case of $\beta=0$, it follows from \eqref{eq:3.8} that there exists a constant $\delta>0$ depending only on $\alpha$ such that for each integer $m\geq 1$,
$$
{\delta\over 2}\int_0^{\sigma_m}|Z_s|^2{\rm d}s\leq \tilde\varphi(\sigma_m,\bar Y_{\sigma_m};1)-\tilde\varphi(0,\bar Y_0;1)+\int_0^{\sigma_m}\tilde\varphi_x(s,\bar Y_s;1){\rm sgn}(Y_s)Z_s\cdot {\rm d}B_s,
$$
where $\bar Y_t:=|Y_t|+\int_0^t|g(s,0,0)|{\rm d}s$ and $\sigma_m$ is a stopping time defined by
$$
\sigma_m:=\inf\left\{ s\in [0,T]:\ \int_0^s \left(\tilde\varphi_x(r,\bar Y_r;1)\right)^2 |Z_r|^2{\rm d}r\geq m\right\}\wedge T.
$$
Then for each real $p>0$, we have
$$
\begin{array}{ll}
&\Dis \left(\int_0^{\sigma_m}|Z_s|^2{\rm d}s\right)^{p\over 2}\vspace{0.2cm}\\
\leq &\Dis \left({4\over \delta}\right)^{p\over 2}\left[\left(\tilde\varphi(\sigma_m,\bar Y_{\sigma_m};1)\right)^{p\over 2}+\sup_{t\in \T}\left|\int_0^{t\wedge\sigma_m}\tilde\varphi_x(s,\bar Y_s;1){\rm sgn}(Y_s)Z_s\cdot {\rm d}B_s\right|^{p\over 2}\right].
\end{array}
$$
In view of inequality $(a+b)^\lambda\leq a^\lambda+b^\lambda$ for $a,b\geq 0$ and $\lambda\in (0,1)$ and H\"{o}lder's inequality together with \eqref{eq:2.13} , \eqref{eq:4.1} and \eqref{eq:4.2}, we get that for each $q>1$,
$$
\begin{array}{ll}
&\Dis \E\left[\left(\sup_{t\in\T}\tilde\varphi(t,\bar Y_t;1)\right)^q\right]\vspace{0.2cm}\\
\leq &\Dis \left(\tilde C_{\alpha,\gamma,1,T}\right)^q\E\left[\psi\left(\sup_{t\in\T}|Y_t|+\int_0^T|g(s,0,0)|{\rm d}s,\  q\tilde\mu_{\alpha,\gamma,1}(T)\right)\right]<+\infty,
\end{array}
$$
where $\tilde\mu_{\alpha,\gamma,1}(\cdot)$ and $\tilde C_{\alpha,\gamma,1,T}$ are respectively defined in \eqref{eq:2.15} and \eqref{eq:4.3}. Note from \eqref{eq:2.5} that for each $s\in \T$ and $x\geq 0$, we have $\tilde\varphi_x(s,x;1)\leq K \tilde\varphi(s,x;1)$ with
$$K:={2(\alpha-1)\tilde\mu_{\alpha,\gamma,1}(T)\over  \alpha k_{\alpha,1}^{{2-\alpha\over \alpha}} },$$
where $k_{\alpha,1}$ is defined in \eqref{eq:2.15}. It follows from the BDG inequality that there exists a constant $C>0$ depending only on $(p,\alpha)$ such that for each $m\geq 1$,
$$
\begin{array}{ll}
& \Dis \left({4\over \delta}\right)^{p\over 2}\E\left[\sup_{t\in \T}\left|\int_0^{t\wedge\sigma_m}\tilde\varphi_x(s,\bar Y_s;1){\rm sgn}(Y_s)Z_s\cdot {\rm d}B_s\right|^{p\over 2}\right]\vspace{0.1cm}\\
\leq & \Dis C
\E\left[\left(K\sup_{t\in\T}\tilde\varphi(t,\bar Y_t;1)\right)^{p\over 2}   \left(\int_0^{\sigma_m}|Z_s|^2{\rm d}s\right)^{p\over 4}\right]\vspace{0.1cm}\\
\leq & \Dis {1\over 2}\E\left[\left(\int_0^{\sigma_m}|Z_s|^2{\rm d}s\right)^{p\over 2}\right]+ {C^2K^p\over 2}
\E\left[\left(\sup_{t\in\T}\tilde\varphi(t,\bar Y_t;1)\right)^p \right].
\end{array}
$$
Combining the previous three inequalities yields the existence of a constant $\bar C>0$ depending only on $(p,\alpha,\gamma,T)$ such that for each $m\geq 1$,
$$
\E\left[\left(\int_0^{\sigma_m}|Z_s|^2{\rm d}s\right)^{p\over 2}\right]\leq \bar C \E\left[\psi\left(\sup_{t\in\T}|Y_t|+\int_0^T|g(s,0,0)|{\rm d}s,\  p\tilde\mu_{\alpha,\gamma,1}(T)\right)\right]<+\infty,
$$
from which the conclusion that $Z\in \mcal^p$ for all $p>0$ follows using Fatou's lemma.

Finally, the uniqueness part is a direct consequence of the following comparison theorem---\cref{thm:4.2}. \cref{thm:4.1} is then proved.
\end{proof}

Let us turn to the comparison theorem of the unbounded solutions.

\begin{thm}\label{thm:4.2}
Let $\xi$ and $\xi'$ be two terminal conditions, $g$ and $g'$ be two generators which are continuous with respect to the state variables $(y,z)$, and $(Y_t, Z_t)_{t\in\T}$ and $(Y'_t, Z'_t)_{t\in\T}$ be respectively a solution to BSDE$(\xi, g)$ and BSDE$(\xi', g')$ such that
\begin{equation}\label{eq:4.6}
\RE p>0,\ \ \ \E\left[\psi\left(\sup\limits_{t\in\T}(|Y_t|+|Y'_t|)+\int_0^T
\left(|g(t,0,0)|+|g'(t,0,0)|\right){\rm d}t,\ p\right)\right]<+\infty.
\end{equation}

Assume that $\ps$, $\xi\leq \xi'$. If $g$ (resp. $g'$) verifies assumptions \ref{H1} and \ref{H2}, and
\begin{equation}\label{eq:4.7}
\as,\ \ \ g(t,Y'_t,Z'_t)\leq g'(t,Y'_t,Z'_t)\ \ \ ({\rm resp.}\  \ g(t,Y_t,Z_t)\leq g'(t,Y_t,Z_t)\ ),
\end{equation}
then $\ps$, for each $t\in\T$, $Y_t\leq Y'_t$.
\end{thm}

\begin{proof}
We first consider the case that the generator $g$ satisfies \ref{H1} with parameters $\alpha$, $\beta$ and $\gamma$, and is convex in $(y,z)$, and $\as$, $g(t,Y'_t,Z'_t)\leq g'(t,Y'_t,Z'_t)$. In order to utilize the convexity of $g$, we use the $\theta$-technique developed in \citet{BriandHu2008PTRF}. For each fixed $\theta\in (0,1)$, define
\begin{equation}\label{eq:4.8}
\delta_\theta  U_\cdot:=\frac{Y_\cdot-\theta Y'_\cdot}{1-\theta}\ \  {\rm and} \ \ \delta_\theta  V_\cdot:=\frac{Z_\cdot-\theta Z'_\cdot}{1-\theta}.
\end{equation}
Then the pair $(\delta_\theta  U_\cdot,\delta_\theta  V_\cdot)$ verifies the following BSDE:
\begin{equation}\label{eq:4.9}
  \delta_\theta  U_t=\delta_\theta  U_T +\int_t^T \delta_\theta g (s,\delta_\theta  U_s,\delta_\theta  V_s) {\rm d}s-\int_t^T \delta_\theta  V_s \cdot {\rm d}B_s, \ \ \ \ t\in\T,
\end{equation}
where $\ass$, for each $(y,z)\in \R\times\R^d$,
\begin{equation}\label{eq:4.10}
\begin{array}{lll}
\Dis \delta_\theta g(s,y,z)&:=& \Dis \frac{1}{1-\theta}\left[\  g(s,(1-\theta)y+\theta Y'_s,(1-\theta)z+\theta Z'_s)-\theta g(s, Y'_s, Z'_s)\ \right]\vspace{0.2cm}\\
&& \Dis +\frac{\theta}{1-\theta}\left[\  g(s,Y'_s, Z'_s)-g'(s,Y'_s, Z'_s)\ \right].
\end{array}
\end{equation}
It follows from the assumptions that $\ass$, for each $(y,z)\in \R\times \R^d$,
\begin{equation}\label{eq:4.11}
\delta_\theta g(s,y,z){\bf 1}_{y>0}\leq g(s,y,z){\bf 1}_{y>0}\leq |g(s,0,0)|+\beta |y|+\gamma |z|^\alpha,
\end{equation}
which means that the generator $\delta_\theta g$ satisfies assumption \ref{H1'} defined in \cref{rmk:3.4} . Thus, in view of \eqref{eq:4.6} and \eqref{eq:4.11} and by virtue of \cref{rmk:3.4} together with the proof of \cref{Pro:3.3}, we can conclude for BSDE \eqref{eq:4.9} that $\ps$, for each $t\in\T$, the inequality
\begin{equation}\label{eq:4.12}
\psi\left(\delta_\theta  U_t^+,\ 1\right)\leq \tilde C_{\alpha,\gamma,1,T}\E\left[\left.\psi\left(\delta_\theta  U_T^+ +\int_0^T|g(s,0,0)|{\rm d}s,\  \tilde\mu_{\alpha,\gamma,1}(T)\right)\right|\F_t\right]
\end{equation}
holds for $\beta=0$, and the inequality
\begin{equation}\label{eq:4.13}
\psi\left(\delta_\theta  U_t^+,\ 1\right)\leq
\bar C_{\alpha,\beta,\gamma,1,T}\E\left[\left.\psi\left(\delta_\theta  U_T^+ +\int_0^T|g(s,0,0)|{\rm d}s, \ \bar\mu_{\alpha,\beta,\gamma,1}(T)\right)\right|\F_t\right]
\end{equation}
holds for $\beta>0$, where $\tilde\mu_{\alpha,\gamma,1}(\cdot)$, $\bar\mu_{\alpha,\beta,\gamma,1}(\cdot)$,
$\tilde C_{\alpha,\gamma,1,T}$ and $\bar C_{\alpha,\beta,\gamma,1,T}$ are respectively defined in \eqref{eq:2.15}, \eqref{eq:2.16} and \eqref{eq:4.3}.
Moreover, in view of the fact that
\begin{equation}\label{eq:4.14}
\delta_\theta  U_T^+=\frac{(\xi-\theta \xi')^+}{1-\theta}=\frac{\left[\xi-\theta \xi+\theta(\xi-\xi')\right]^+}{1-\theta}\leq \xi^+,
\end{equation}
by \eqref{eq:4.12} and \eqref{eq:4.13} we derive that $\ps$, for each $t\in\T$, for $\beta=0$,
$$
\left(Y_t-\theta Y'_t\right)^+ \leq (1-\theta)
\left(\ln\left\{\tilde C_{\alpha,\gamma,1,T}\E\left[\left.\psi\left(\xi^+ +\int_0^T|g(s,0,0)|{\rm d}s, \ \tilde\mu_{\alpha,\gamma,1}(T)\right)\right|\F_t\right]\right\}
\right)^{\frac{\alpha^*}{2}},
$$
and for $\beta>0$,
$$
\left(Y_t-\theta Y'_t\right)^+\leq (1-\theta)
\left(\ln\left\{\bar C_{\alpha,\beta,\gamma,1,T} \E\left[\left.\psi\left(\xi^+ +\int_0^T|g(s,0,0)|{\rm d}s, \ \bar\mu_{\alpha,\beta,\gamma,1}(T)\right)\right|\F_t\right]
\right\}\right)^{\frac{\alpha^*}{2}}.
$$
Consequently, the desired conclusion follows by sending $\theta\To 1$ in the previous two\vspace{0.2cm} inequalities.

For the case that the generator $g$ is concave with respect to the state variables $(y,z)$, we need to use $\theta Y_\cdot-Y'_\cdot$ and $\theta Z_\cdot-Z'_\cdot$ instead of $Y_\cdot-\theta Y'_\cdot$ and $Z_\cdot-\theta Z'_\cdot$ in \eqref{eq:4.8} respectively. And, in this case the generator $\delta_\theta  g$ in \eqref{eq:4.10} should be replaced with
\begin{equation}\label{eq:4.15}
\begin{array}{lll}
\Dis \delta_\theta g(s,y,z)&:=& \Dis \frac{1}{1-\theta}\left[\  \theta g(s, Y_s, Z_s) -g(s,-(1-\theta)y+\theta Y_s,-(1-\theta)z+\theta Z_s)\ \right]\vspace{0.2cm}\\
&& \Dis +\frac{1}{1-\theta}\left[\  g(s,Y'_s, Z'_s)-g'(s,Y'_s, Z'_s)\ \right].
\end{array}
\end{equation}
Since $g$ is concave in $(y,z)$, we have, $\ass$, for each $(y,z)\in \R\times \R^d$,
$$
g(s,-(1-\theta)y+\theta Y_s,-(1-\theta)z+\theta Z_s)\geq \theta g(s,Y_s,Z_s)+(1-\theta)g(t,-y,-z),
$$
and then, \eqref{eq:4.11} needs to be replaced by
\begin{equation}\label{eq:4.16}
\delta_\theta g(s,y,z){\bf 1}_{y>0}\leq -g(s,-y,-z){\bf 1}_{y>0}\leq |g(s,0,0)|+\beta |y|+\gamma |z|^\alpha,
\end{equation}
which means that the generator $\delta_\theta g$ still satisfies assumption \ref{H1'}. Consequently, both \eqref{eq:4.12} and \eqref{eq:4.13} still hold.
Moreover, we use
\[
\delta_\theta U_T^+=\frac{(\theta \xi-\xi')^+}{1-\theta}=\frac{\left[\theta \xi- \xi+(\xi-\xi')\right]^+}{1-\theta}\leq (-\xi)^+=\xi^-,
\]
instead of \eqref{eq:4.14}, and by virtue of \eqref{eq:4.12} and \eqref{eq:4.13}, derive that $\ps$, for each $t\in\T$, for $\beta=0$,
$$
\left(\theta Y_t-Y'_t\right)^+ \leq (1-\theta)
\left(\ln\left\{\tilde C_{\alpha,\gamma,1,T}\E\left[\left.\psi\left(\xi^- +\int_0^T|g(s,0,0)|{\rm d}s, \ \tilde\mu_{\alpha,\gamma,1}(T)\right)\right|\F_t\right]\right\}
\right)^{\frac{\alpha^*}{2}},
$$
and for $\beta>0$,
$$
\left(\theta Y_t-Y'_t\right)^+\leq (1-\theta)
\left(\ln\left\{\bar C_{\alpha,\beta,\gamma,1,T} \E\left[\left.\psi\left(\xi^- +\int_0^T|g(s,0,0)|{\rm d}s, \ \bar\mu_{\alpha,\beta,\gamma,1}(T)\right)\right|\F_t\right]
\right\}\right)^{\frac{\alpha^*}{2}}.
$$
Thus, the desired conclusion follows by sending $\theta\To 1$ in the previous two inequalities.\vspace{0.1cm}

Finally, in the same way as above, one can prove the desired conclusion under the conditions that the generator $g'$ satisfies assumptions \ref{H1} and \ref{H2}, and $\as$, $g(t,Y_t,Z_t)\leq g'(t,Y_t,Z_t)$. The proof of \cref{thm:4.2} is then complete.
\end{proof}

\begin{rmk}\label{rmk:4.3}
Clearly, if $\as$, for each $(y,z)\in \R\times\R^d$, $g(t,y,z)\leq g'(t,y,z)$, then the inequality \eqref{eq:4.7} holds.
\end{rmk}

\begin{rmk}\label{rmk:4.4}
From the above proofs, it is not hard to verify that the assumption \ref{H1} in \cref{thm:4.1} and \cref{thm:4.2} can also be relaxed to the weaker assumption \ref{H1"} defined in \cref{rmk:3.5}.
\end{rmk}

\section{An extension to the comparison theorem}
\label{sec:5}
\setcounter{equation}{0}

In this section, we first introduce a general non-convexity (non-convexity) assumption \ref{H2'} on the generator $g$, and then illustrate that it is strictly weaker than the assumption \ref{H2} provided that the assumption \ref{H1"} or \ref{H1} holds for $g$. Finally, we prove that Theorems \ref{thm:4.1} and \ref{thm:4.2} hold still under the weaker assumptions \ref{H1"} and \ref{H2'}.
Let us start by introducing assumption \ref{H2'}:
\begin{enumerate}
\renewcommand{\theenumi}{(H2')}
\renewcommand{\labelenumi}{\theenumi}
\item\label{H2'} There exist four real constants $\alpha\in (1,2)$, $\beta\geq 0$, $\gamma>0$ and $k>0$, and a progressively measurable $\R_+$-valued process $(f_t)_{t\in \T}$ such that $\as$, for each $(y_i,z_i)\in \R\times\R^d$, $i=1,2$ and each $\theta\in (0,1)$, it holds that
  \begin{equation}\label{eq:5.1}
  \begin{array}{ll}
  &\Dis {\bf 1}_{\{y_1-\theta y_2>0\}}\left(g(\omega,t,y_1,z_1)-\theta g(\omega,t,y_2,z_2)\right)\vspace{0.1cm}\\
   \leq &\Dis (1-\theta)\left(f_t(\omega)+k|y_2|+\beta \left|\delta_\theta y\right|+\gamma \left|\delta_\theta z\right|^\alpha\right)
   \end{array}
  \end{equation}
   or
  \begin{equation}\label{eq:5.2}
  \begin{array}{ll}
  & \Dis -{\bf 1}_{\{y_1-\theta y_2<0\}}\left(g(\omega,t,y_1,z_1)-\theta g(\omega,t,y_2,z_2)\right)\vspace{0.1cm}\\
   \leq &\Dis (1-\theta)\left(f_t(\omega)+k|y_2|+\beta \left|\delta_\theta y\right|+\gamma \left|\delta_\theta z\right|^\alpha\right),
  \end{array}
  \end{equation}
  where
  $$
  \delta_\theta y:=\frac{y_1-\theta y_2}{1-\theta},\ \ \ \delta_\theta z:=\frac{z_1-\theta z_2}{1-\theta}.\vspace{0.2cm}
  $$
\end{enumerate}

One typical example of \ref{H2'} is
$g(\omega,t,y,z):=g_1(y)+g_2(y)$,
where $g_1:\R\To\R$ is convex or concave with one-sided linear growth, and $g_2:\R\To\R$ is a Lipschitz function, i.e., $g$ is a Lipschitz perturbation of some convex (concave) function.\vspace{0.2cm}

Another typical example of \ref{H2'} is
$\bar g(\omega,t,y,z):=g_3(z)+g_4(z)$,
where $g_3:\R^d\To\R$ is convex or concave with sub-quadratic growth, and $g_4:\R^d\To\R$ is a Lipschitz funtion with bounded support, i.e., $\bar g$ is a locally Lipschitz perturbation of some convex (concave) function.\vspace{0.2cm}

More generally, we have

\begin{pro}\label{pro:5.1}
Assume that the generator $g$ is continuous in $(y,z)$ and
satisfies assumption \ref{H1"}. Then, assumption \ref{H2'} holds for $g$ if it satisfies anyone of the following conditions:
\begin{enumerate}
\item [(i)] $\as$, $g(\omega,t,\cdot,\cdot)$ is convex or concave;

\item [(ii)]$\as$, for each $(y,z)\in \R\times\R^d$, $g(\omega,t,\cdot,z)$ is Lipschitz and $g(\omega,t, y,\cdot)$ is convex or concave;

\item [(iii)] $g(t,y,z)\equiv l(y)q(z)$, where both $l:\R\To\R$ and $q:\R^d\To\R$ are bounded Lipschitz functions, and the function $q(z)$ has a bounded support.
\end{enumerate}
\end{pro}

Before giving the proof of this proposition, we first make the following important remark.

\begin{rmk}\label{rmk:5.2}
It is easy to verify that if for $i=1,2$, the generator $g_i$ is continuous in $(y,z)$ and satisfies assumption \ref{H1"} together with anyone of (i), (ii) and (iii) in \cref{pro:5.1} (with the same convexity or concavity when available), then $g_1+g_2$ also satisfies assumption \ref{H2'}. Consequently, the generator $g$ satisfying \ref{H2'} may be not necessarily convex (concave) or Lipschitz in the variables $y$ and $z$, and it can have a general growth in $y$.
\end{rmk}

\begin{proof} [Proof of \cref{pro:5.1}] Given $(y_i,z_i)\in \R\times\R^d$, $i=1,2$ and $\theta\in (0,1)$.\vspace{0.2cm}

(i) Assume that $\as$, $g(\omega,t,\cdot,\cdot)$ is convex. In view of \ref{H1"}, if $\delta_\theta y>0$, then
$$
\begin{array}{lll}
g(\omega,t,y_1,z_1)&=&\Dis g\left(\omega,t,\theta y_2+(1-\theta)\delta_\theta y,\theta z_2+(1-\theta)\delta_\theta z\right)\vspace{0.1cm}\\
&\leq & \Dis \theta g(\omega,t,y_2,z_2)+(1-\theta)g\left(\omega,t,\delta_\theta y,\delta_\theta z\right)\vspace{0.1cm}\\
&\leq & \Dis \theta g(\omega,t,y_2,z_2)+(1-\theta)\left(f_t(\omega)+\beta |\delta_\theta y|+\gamma |\delta_\theta z|^\alpha\right).
\end{array}
$$
Thus, the inequality \eqref{eq:5.1} holds with $k=0$. The concave case is similar.\vspace{0.2cm}

(ii) Assume that $\as$, for each $(y,z)\in \R\times\R^d$, $g(\omega,t,\cdot,z)$ is Lipschitz with Lipschitz constant $\beta$, and $g(\omega,t, y,\cdot)$ is convex. Then, noticing by \ref{H1"} that $|g(\omega,t,0,z)|\leq f_t+\gamma |z|^2$, we have
$$
\begin{array}{ll}
&\Dis g(\omega,t,y_1,z_1)-\theta g(\omega,t,y_2,z_2)\vspace{0.1cm}\\
\leq & \Dis |g(\omega,t,y_1,z_1)-g(\omega,t,y_2,z_1)|+g(\omega,t,y_2,z_1)-\theta g(\omega,t,y_2,z_2)\vspace{0.1cm}\\
\leq & \Dis \beta |y_1-y_2|+g(\omega,t,y_2,\theta z_2+(1-\theta)\delta_\theta z)-\theta g(\omega,t,y_2,z_2)
\vspace{0.1cm}\\
\leq & \Dis \beta (|y_1-\theta y_2|+(1-\theta)|y_2|)+(1-\theta)\left(|g(\omega,t,y_2,\delta_\theta z)-g(\omega,t,0,\delta_\theta z)|+|g(\omega,t,0,\delta_\theta z)|\right)\vspace{0.1cm}\\
\leq & \Dis (1-\theta)\left(\beta |\delta_\theta y|+2\beta |y_2|+f_t(\omega)+\gamma |\delta_\theta z|^\alpha\right).
\end{array}
$$
Thus, \eqref{eq:5.1} holds with $2\beta$ instead of $k$. The concave case is similar.\vspace{0.2cm}

(iii) With loss of generality, we assume that the functions $l(y)$ and $q(z)$ have Lipschitz constants $\beta$ and $\gamma$ together with a same bound $M>0$, and $q(z)\equiv 0$ when $|z|>R$ for some $R>0$. Noticing that
$$
\begin{array}{lll}
|q(\theta z_2)-q(z_2)|&=& \Dis |q(\theta z_2)-q(z_2)|{\bf 1}_{\theta\in (0,1/2]}+|q(\theta z_2)-q(z_2)|{\bf 1}_{\theta\in (1/2,1)}\vspace{0.1cm}\\
&\leq & \Dis (1-\theta)\frac{2M}{1-\theta}{\bf 1}_{\theta\in (0,1/2]}+(1-\theta)\gamma |z_2|{\bf 1}_{|z_2|\leq 2R}|{\bf 1}_{\theta\in (1/2,1)}\vspace{0.1cm}\\
&\leq & \Dis (1-\theta)(4M+2\gamma R),\vspace{-0.1cm}
\end{array}
$$
we have\vspace{-0.1cm}
$$
\begin{array}{ll}
& \Dis g(\omega,t,y_1,z_1)-\theta g(\omega,t,y_2,z_2)=l(y_1)q(z_1)-\theta l(y_2)q(z_2)\vspace{0.1cm}\\
\leq & \Dis |l(y_1)-l(y_2)||q(z_1)|+|l(y_2)||q(z_1)-\theta q(z_2)|\vspace{0.1cm}\\
\leq & \Dis M\beta |y_1-y_2|+M\left(|q(z_1)-q(\theta z_2)|+|q(\theta z_2)-q(z_2)|+(1-\theta)|q(z_2)|\right)\vspace{0.1cm}\\
\leq & \Dis M\beta (|y_1-\theta y_2|+(1-\theta)|y_2|)+M\left(\gamma |z_1-\theta z_2|+(1-\theta)(4M+2\gamma R)+(1-\theta)M\right)\vspace{0.1cm}\\
\leq & \Dis (1-\theta)M\left(\beta |\delta_\theta y|+\beta |y_2|+\gamma |\delta_\theta z|+5M+2\gamma R\right)\vspace{0.1cm}\\
\leq & \Dis (1-\theta)M\left(5M+2\gamma R+\gamma+\beta |y_2|+\beta |\delta_\theta y|+\gamma |\delta_\theta z|^\alpha\right).\vspace{0.1cm}
\end{array}
$$
Thus, the inequality \eqref{eq:5.1} holds with $M(5M+2\gamma R+\gamma)$ instead of $f_\cdot$, $M\beta$ instead of $k$ and $\beta$, and $M\gamma$ instead of $\gamma$ respectively. The proposition is then proved.
\end{proof}

\begin{rmk}\label{rmk:5.3}
(i) Letting $y_1=y_2=y$ and $z_1=z_2=z$ in \eqref{eq:5.1} and \eqref{eq:5.2} respectively yields that
$$
{\bf 1}_{\{y>0\}} g(\omega,t,y,z)\leq f_t(\omega)+(\beta+k)|y|+\gamma |z|^\alpha
$$
and
$$
-{\bf 1}_{\{y<0\}} g(\omega,t,y,z)\leq f_t(\omega)+(\beta+k)|y|+\gamma |z|^\alpha,
$$
whose combination implies that $g$ has a one-sided linear growth in the state variable $y$ and a sub-quadratic growth in the state variable $z$.\vspace{0.2cm}

(ii) Letting first $z_1=z_2=z$ in \eqref{eq:5.1} and \eqref{eq:5.2} and then letting $\theta\To 1$ yields that
$$
{\bf 1}_{\{y_1-y_2>0\}} \left(g(\omega,t,y_1,z)-g(\omega,t,y_2,z)\right)\leq \beta |y_1-y_2|,
$$
which means that $g$ satisfies the so-called monotonicity condition in $y$.\vspace{0.2cm}

(iii) Set $y=\frac{y_1-\theta y_2}{1-\theta}$, $z=\frac{z_1-\theta z_2}{1-\theta}$, $\bar y=y_2$ and $\bar z=z_2$. Then, \eqref{eq:5.1} and \eqref{eq:5.2} can be respectively rewritten as the following forms:
\begin{equation}\label{eq:5.3}
  \begin{array}{ll}
  &\Dis {\bf 1}_{\{y>0\}}\left(g(\omega,t,(1-\theta)y+\theta \bar y,(1-\theta)z+\theta \bar z)-\theta g(\omega,t,\bar y,\bar z)\right)\vspace{0.1cm}\\
   \leq &\Dis (1-\theta)\left(f_t(\omega)+k|\bar y|+\beta |y|+\gamma |z|^\alpha\right)
   \end{array}
  \end{equation}
and
\begin{equation}\label{eq:5.4}
  \begin{array}{ll}
  &\Dis -{\bf 1}_{\{y<0\}}\left(g(\omega,t,(1-\theta)y+\theta \bar y,(1-\theta)z+\theta \bar z)-\theta g(\omega,t,\bar y,\bar z)\right)\vspace{0.1cm}\\
   \leq &\Dis (1-\theta)\left(f_t(\omega)+k|\bar y|+\beta |y|+\gamma |z|^\alpha\right).\vspace{0.4cm}
  \end{array}
  \end{equation}
\end{rmk}

The following \cref{thm:5.4} and \cref{thm:5.5} are the main results of this section.

\begin{thm}\label{thm:5.4}
Let $\xi$ and $\xi'$ be two terminal conditions such that $\ps$, $\xi\leq \xi'$, $g$ and $g'$ be two generators which are continuous in $(y,z)$, $g$ (resp. $g'$) verifies assumption \ref{H2'} with constants $(\alpha,\beta,\gamma,k)$ and process $f_\cdot$, and $(Y_t, Z_t)_{t\in\T}$ and $(Y'_t, Z'_t)_{t\in\T}$ be respectively a solution to BSDE$(\xi, g)$ and BSDE$(\xi', g')$ such that
\begin{equation}\label{eq:5.5}
\RE p>0,\ \ \ \E\left[\psi\left(\sup\limits_{t\in\T}(|Y_t|+|Y'_t|)+\int_0^T
f_t{\rm d}t,\ p\right)\right]<+\infty.
\end{equation}

If $\as$, we have
\[
g(t,Y'_t,Z'_t)\leq g'(t,Y'_t,Z'_t)\ \ \ ({\rm resp.}\  \ g(t,Y_t,Z_t)\leq g'(t,Y_t,Z_t)\ ),
\]
then $\ps$, for each $t\in\T$, $Y_t\leq Y'_t$.
\end{thm}

\begin{proof}
We only prove the case that the generator $g$ satisfies \eqref{eq:5.1} with constants $(\alpha,\beta,\gamma,k)$ and process $f_\cdot$, and $\as$, $g(t,Y'_t,Z'_t)\leq g'(t,Y'_t,Z'_t)$. For each fixed $\theta\in (0,1)$, with the notations in \eqref{eq:4.8}, we know that the pair $(\delta_\theta  U_\cdot,\delta_\theta  V_\cdot)$ verifies BSDE \eqref{eq:4.9} with generator $g$ defined in \eqref{eq:4.10}. Then, in view of \eqref{eq:5.3}, it follows from the assumptions that $\ass$, for each $(y,z)\in \R\times \R^d$, we have
\begin{equation}\label{eq:5.6}
\delta_\theta g(s,y,z){\bf 1}_{y>0}\leq |f_s|+k|Y'_s|+\beta |y|+\gamma |z|^\alpha,
\end{equation}
which means that the generator $\delta_\theta g$ satisfies assumption \ref{H1'} with the process $|f_\cdot|+k|Y'_\cdot|$ instead of $|g(\cdot,0,0)|$. Thus, thanks to \eqref{eq:5.5}, the rest of proof runs as that in \cref{thm:4.2}.\vspace{0.1cm}
\end{proof}

\begin{thm}\label{thm:5.5}
Assume that $\xi$ is a terminal condition and $g$ is a generator which is continuous in $(y,z)$ and satisfies assumptions \ref{H1"} and \ref{H2'} with constants $(\alpha,\beta,\gamma,k)$ and process $f_\cdot$.\vspace{0.1cm}

If the terminal condition $\xi+\int_0^T f_t {\rm d}t$ has sub-exponential moments of any order, i.e., \vspace{0.1cm}
\[
\RE\ p>0,\ \ \E\left[\psi\left(|\xi|+\int_0^T f_t {\rm d}t,\ p\right)\right]=\E\left[\exp\left\{p \left(|\xi|+\int_0^T f_t {\rm d}t\right)^{\frac{2}{\alpha^*}} \right\}\right]<+\infty,\vspace{0.1cm}
\]
then BSDE$(\xi,g)$ admits a unique solution $(Y_t,Z_t)_{t\in\T}$ such that $\sup_{t\in\T}|Y_t|$ has sub-exponential moments of any order (i.e., \eqref{eq:4.2} holds). Furthermore, $Z_\cdot\in \mcal^p$ for all $p>0$.
\end{thm}

\begin{proof}
In view of \cref{rmk:3.5} and \cref{thm:5.4}, the proof runs as that in \cref{thm:4.1}.
\end{proof}

\begin{rmk}\label{rmk:5.6}
From \cref{rmk:3.5}, \cref{pro:5.1} and \cref{rmk:5.2}, it is not difficult to see that Theorems \ref{thm:5.4} and \ref{thm:5.5} respectively extend Theorems \ref{thm:4.2} and \ref{thm:4.1} to the non-convexity and non-concavity case.
\end{rmk}

\section{A stabilty theorem of the solutions}
\label{sec:6}
\setcounter{equation}{0}

In this section, we establish the following stability result for the unbounded solutions of BSDEs under general assumptions \ref{H1"} and \ref{H2'}.

\begin{thm}\label{thm:6.1}
Assume that $\xi$ is a terminal condition, $g$ is a generator which is continuous in $(y,z)$ and satisfies assumptions \ref{H1"} and \ref{H2'} with constants $(\alpha,\beta,\gamma,k)$ and process $f_\cdot$, and $(Y_t,Z_t)_{t\in\T}$ is the (unique) solution to BSDE$(\xi,g)$ such that $\sup_{t\in\T}|Y_t|$ has sub-exponential moments of any order.\vspace{0.1cm}

Assume also that for each $n\geq 1$, $\xi^n$ is a terminal condition, $g^n$ is a generator which is continuous in $(y,z)$ and satisfies assumptions \ref{H1"} and \ref{H2'} with constants $(\alpha,\beta,\gamma,k)$ and process $f^n_\cdot$, and $(Y^n_t,Z^n_t)_{t\in\T}$ is the (unique) solution to BSDE$(\xi^n,g^n)$ such that $\sup_{t\in\T}|Y^n_t|$ has sub-exponential moments of any order.\vspace{0.1cm}

Let us assume further that for each $p>0$,
\begin{equation}\label{eq:6.1}
\E\left[\exp\left\{p \left(|\xi|+\int_0^T f_t {\rm d}t\right)^{\frac{2}{\alpha^*}} \right\}\right]+\sup_{n\geq 1}\E\left[\exp\left\{p \left(|\xi^n|+\int_0^T f^n_t {\rm d}t\right)^{\frac{2}{\alpha^*}} \right\}\right]<+\infty.
\end{equation}
If $\ps$, $\xi^n\To \xi$ as $n\To \infty$ and there exists a real $q>1$ such that
\begin{equation}\label{eq:6.2}
\Lim \E\left[\left(\int_0^T \left| g^n(s,Y_s, Z_s)-g(s,Y_s, Z_s)\right| {\rm d}s\right)^q\right]=0,
\end{equation}
then for each $p>0$, we have
\begin{equation}\label{eq:6.3}
\Lim \E\left[\exp\left\{p \left(\sup_{t\in\T}|Y^n_t-Y_t|\right)^{\frac{2}{\alpha^*}}\right\}\right]=1.
\end{equation}
And, if the function $h(\cdot)$ defined in \ref{H1"} further satisfies that for some constant $c>0$,
\begin{equation}\label{eq:6.4}
h(|x|)\leq c\exp (c|x|^{2\over \alpha^*}),\ \ x\in \R,
\end{equation}
then for each $p>0$, we have
\begin{equation}\label{eq:6.5}
\Lim \E\left[\left(\int_0^T|Z^n_s-Z_s|^2{\rm d}s\right)^{p\over 2}
\right]=0.\vspace{0.2cm}
\end{equation}
\end{thm}

\begin{proof}
It follows from the integrability assumption \eqref{eq:6.1} and the proof of \cref{thm:5.5} and \cref{thm:4.1} (see, in particular, inequality \eqref{eq:4.5}) that the sequence $(Y^n_t,Z^n_t)_{t\in\T}$ satisfies
\begin{equation}\label{eq:6.6}
\RE\ p>0, \ \ \sup_{n\geq 1} \E\left[\exp\left\{p \left(\sup_{t\in\T}|Y^n_t|\right)^{\frac{2}{\alpha^*}} \right\}
+\left(\int_0^T|Z^n_s|^2{\rm d}s\right)^{p\over 2}
\right]<+\infty.\vspace{0.1cm}
\end{equation}
It is thus enough to prove that
$$
\sup_{t\in\T}|Y^n_t-Y_t|+\int_0^T|Z^n_s-Z_s|^2{\rm d}s
$$
converges to $0$ in probability to get the desired conclusion.\vspace{0.2cm}

We only prove the case that $\beta=0$ and inequality \eqref{eq:5.1} holds for $g$ and $g^n$. The other cases can be proved in the same way. For each fixed $\theta\in (0,1)$, define
$$
\delta_\theta^n  U_\cdot:=\frac{Y^n_\cdot-\theta Y_\cdot}{1-\theta}\ \  {\rm and} \ \ \delta_\theta^n  V_\cdot:=\frac{Z^n_\cdot-\theta Z_\cdot}{1-\theta}.
$$
Then the pair $(\delta_\theta^n  U_\cdot,\delta_\theta^n  V_\cdot)$ verifies the following BSDE:
\begin{equation}\label{eq:6.7}
  \delta_\theta^n  U_t=\delta_\theta^n  U_T +\int_t^T \delta_\theta^n g (s,\delta_\theta^n  U_s,\delta_\theta^n  V_s) {\rm d}s-\int_t^T \delta_\theta^n  V_s \cdot {\rm d}B_s, \ \ \ \ t\in\T,
\end{equation}
where $\ass$, for each $(y,z)\in \R\times\R^d$,
\begin{equation}\label{eq:6.8}
\delta_\theta^n g(s,y,z):=\Dis \frac{1}{1-\theta}\left(g^n(s,(1-\theta)y+\theta Y_s,(1-\theta)z+\theta Z_s)-\theta g^n(s, Y_s, Z_s)\right)+D^n_\theta(s)
\end{equation}
with
$$
D^n_\theta(s):=\frac{\theta}{1-\theta}\left( g^n(s,Y_s, Z_s)-g(s,Y_s, Z_s) \right).
\vspace{0.1cm}
$$
Since \eqref{eq:5.1} holds for $g^n$, it follows from \eqref{eq:6.8} and \eqref{eq:5.3} that $\ass$, for each $(y,z)\in \R\times \R^d$,
\begin{equation}\label{eq:6.9}
\delta_\theta^n g(s,y,z){\bf 1}_{y>0}\leq f^n_s+k|Y_s|+|D^n_\theta(s)|+\gamma |z|^\alpha.
\end{equation}
Now, let the functions $\tilde\varphi(s,x;\eps)$ and $\psi(x,\mu)$ be defined respectively in \eqref{eq:2.13} and \eqref{eq:3.1}, and denote
$$
\Delta_\theta^n  U_t:=(\delta_\theta^n  U_t)^+ +\int_0^t \left(f^n_s+k|Y_s|\right){\rm d}s,\ \ t\in\T.\vspace{0.1cm}
$$
Applying It\^{o}-Tanaka's formula to the process $\tilde\varphi(s, \Delta_\theta^n  U_s; 1)$ and using \eqref{eq:6.7}, \eqref{eq:6.9} and \eqref{eq:2.18} in \cref{Pro:2.1}, a similar computation to that in the proof of \cref{Pro:3.3} yields the existence of a positive constant $\bar\delta>0$ depending only on $\alpha$ such that
\begin{equation}\label{eq:6.10}
\begin{array}{lll}
{\rm d}\tilde\varphi(s,\Delta_\theta^n  U_s;1)& \geq & \Dis -\tilde\varphi_x(s,\Delta_\theta^n  U_s;1)|D^n_\theta(s)|{\rm d}s+ \frac{\bar\delta}{2}{\bf 1}_{\delta_\theta^n  U_s>0} |\delta_\theta^n  V_s|^2 {\rm d}s\vspace{0.2cm}\\
&&\Dis+\tilde\varphi_x(s,\Delta_\theta^n  U_s;1){\bf 1}_{\delta_\theta^n  U_s>0} \delta_\theta^n  V_s \cdot {\rm d}B_s,\ \ s\in \T.\vspace{0.1cm}
\end{array}
\end{equation}
Note from \eqref{eq:2.5} that for each $s\in \T$ and $x\geq 0$, we have $\tilde\varphi_x(s,x;1)\leq K \tilde\varphi(s,x;1)$ for some positive constant $K$ depending only on $(\alpha,\gamma,T)$. It follows from the BDG inequality, Young's inequality, \eqref{eq:6.1} and \eqref{eq:6.6} that the process
$$
\left(\int_0^t \tilde\varphi_x(s,\Delta_\theta^n  U_s;1){\bf 1}_{\delta_\theta^n  U_s>0} \delta_\theta^n  V_s \cdot {\rm d}B_s\right)_{t\in\T}
$$
is a uniformly integrable martingale. Then, from \eqref{eq:6.10} we know that
$$
\begin{array}{ll}
& \Dis \tilde\varphi(t,\Delta_\theta^n  U_s;1)+\frac{\bar\delta }{2}\E\left[\left.\int_t^T {\bf 1}_{\delta_\theta^n  U_s>0} |\delta_\theta^n  V_s|^2 {\rm d}s\right|\F_t\right]\vspace{0.2cm}\\
\leq & \Dis \E\left[\left.\tilde\varphi(T,\Delta_\theta^n  U_T;1)\right|\F_t\right] +K \E\left[\left.\int_t^T \tilde\varphi(s,\Delta_\theta^n  U_s;1)|D^n_\theta(s)|{\rm d}s\right|\F_t\right],\ \ t\in \T.
\end{array}
$$
Furthermore, by the definitions of functions $\tilde\varphi$ and $\psi$ together with $\Delta_\theta^n  U_s$ we can conclude that there exists a positive constant $\bar K$ depending only on $(\alpha,\gamma,T,k)$ such that $\ps$,
$$
\begin{array}{ll}
&\Dis \psi\left(\left(\delta_\theta^n  U_t\right)^+,\ 1\right)+\frac{\bar\delta }{2}\E\left[\left.\int_t^T {\bf 1}_{\delta_\theta^n  U_s>0} |\delta_\theta^n  V_s|^2 {\rm d}s\right|\F_t\right]\vspace{0.2cm}\\
\leq &\Dis \bar K \E\left[\left.\psi\left(\Delta_\theta^n  U_T,\  \bar K\right)+ \int_0^T \psi\left(\Delta_\theta^n  U_s,\ \bar K\right)|D^n_\theta(s)|{\rm d}s \right|\F_t\right],\ \ t\in\T.
\end{array}
$$
Consequently, for each $n\geq 1$, $\theta\in (0,1)$ and $t\in\T$, we have
\begin{equation}\label{eq:6.11}
\begin{array}{ll}
&\Dis (Y^n_t-\theta Y_t)^+\vspace{0.1cm}\\
\leq &\Dis  (1-\theta)\left(\bar K\E\left[\left.\psi\left(\Delta_\theta^n  U_T,\  \bar K\right)+ \int_0^T \psi\left(\Delta_\theta^n  U_s,\ \bar K\right)|D^n_\theta(s)|{\rm d}s \right|\F_t\right]\right)^{\alpha^*\over 2}.
\end{array}\vspace{0.2cm}
\end{equation}

On the other hand, for each fixed $\theta\in (0,1)$, we define
$$
\bar\delta_\theta^n  U_\cdot:=\frac{Y_\cdot-\theta Y^n_\cdot}{1-\theta}\ \  {\rm and} \ \ \bar\delta_\theta^n  V_\cdot:=\frac{Z_\cdot-\theta Z^n_\cdot}{1-\theta}.
$$
Then the pair $(\bar\delta_\theta^n  U_\cdot,\bar\delta_\theta^n  V_\cdot)$ verifies the following BSDE:
$$
\bar\delta_\theta^n  U_t=\bar\delta_\theta^n  U_T +\int_t^T \bar\delta_\theta^n g (s,\bar\delta_\theta^n  U_s,\bar\delta_\theta^n  V_s) {\rm d}s-\int_t^T \bar\delta_\theta^n  V_s \cdot {\rm d}B_s, \ \ \ \ t\in\T,
$$
where $\ass$, for each $(y,z)\in \R\times\R^d$,
$$
\bar\delta_\theta^n g(s,y,z):=\Dis \frac{1}{1-\theta}\left(g^n(s,(1-\theta)y+\theta Y^n_s,(1-\theta)z+\theta Z^n_s)-\theta g^n(s, Y^n_s, Z^n_s)\right)+\bar D^n_\theta(s)
$$
with
$$
\bar D^n_\theta(s):=\frac{1}{1-\theta}\left(g(s,Y_s, Z_s) -g^n(s,Y_s, Z_s)\right).
$$
Since \eqref{eq:5.1} holds for $g^n$, it follows that $\ass$, for each $(y,z)\in \R\times \R^d$,
$$
\delta_\theta^n g(s,y,z){\bf 1}_{y>0}\leq f^n_s+k|Y^n_s|+|\bar D^n_\theta(s)|+\gamma |z|^\alpha.
$$
Then, let us denote
$$
\bar \Delta_\theta^n  U_t:=(\bar \delta_\theta^n  U_t)^+ +\int_0^t \left(f^n_s+k|Y^n_s|\right){\rm d}s,\ \ t\in\T.
$$
A similar computation as that from inequality \eqref{eq:6.10} to inequality \eqref{eq:6.11}
yields that for each $n\geq 1$, $t\in \T$ and $\theta\in (0,1)$,
\begin{equation}\label{eq:6.12}
\begin{array}{ll}
&\Dis (Y_t-\theta Y^n_t)^+ \vspace{0.1cm}\\
\leq & \Dis (1-\theta)\left(\bar K\E\left[\left.\psi\left(\bar\Delta_\theta^n  U_T,\  \bar K\right)+ \int_0^T \psi\left(\bar \Delta_\theta^n  U_s,\ \bar K\right)|\bar D^n_\theta(s)|{\rm d}s \right|\F_t\right]\right)^{\alpha^*\over 2}.\vspace{0.3cm}
\end{array}
\end{equation}

In the sequel, combining \eqref{eq:6.11} and \eqref{eq:6.12} together with inequalities
$$
Y^n_t-Y_t \leq (Y^n_t-\theta Y_t)^+ +(1-\theta)|Y_t|\ \ {\rm and}\ \ Y_t-Y^n_t \leq (Y_t-\theta Y^n_t)^+ +(1-\theta)|Y^n_t|,
$$
we can deduce that for each $n\geq 1$, $\theta\in (0,1)$ and $t\in\T$,
\begin{equation}\label{eq:6.13}
\begin{array}{l}
\Dis |Y^n_t-Y_t|\leq (1-\theta)(|Y_t|+|Y^n_t|)+4(1-\theta)\left(\bar K\E\left[\left.\psi\left(X^n (\theta)+G^n_T,\ \bar K\right)\right|\F_t\right] \right)^{\alpha^*\over 2}\vspace{0.3cm}\\
\hspace*{-0.2cm}\Dis +\frac{4}{(1-\theta)^{\alpha^*-2\over 2}}\left\{\bar K\E\left[\left.\psi\left(H^n (\theta)+G^n_T,\ \bar K\right)\int_0^T \left| g^n(s,Y_s, Z_s)-g(s,Y_s, Z_s)\right| {\rm d}s \right|\F_t\right]\right\}^{\alpha^*\over 2},\vspace{0.1cm}
\end{array}
\end{equation}
where
$$
X^n (\theta):=\frac{|\xi-\theta \xi^n| \vee |\xi^n-\theta \xi|}{1-\theta},\ \ H^n (\theta):=\frac{\sup_{t\in\T}\left(|Y_t|+|Y^n_t|\right)}{1-\theta}
$$
and
$$
G^n_T:=\int_0^T \left(f^n_s+k|Y^n_s|+k|Y_s|\right){\rm d}s.\vspace{0.1cm}
$$
Now, let us fix $\eps>0$. It follows from \eqref{eq:6.13} and Doob's maximal inequality on martingale that
$$
\begin{array}{ll}
&\Dis \mathbb{P}\left(\sup_{t\in\T} |Y^n_t-Y_t|>\eps\right)\vspace{0.3cm}\\
\leq &\Dis \frac{3(1-\theta)}{\eps}\E\left[\sup_{t\in \T}(|Y^n_t|+|Y_t|)\right]+\left(\frac{12(1-\theta)}{\eps}\right)^{2\over \alpha^*}\bar K\E\left[\psi\left(X^n (\theta),\ \bar K\right)\psi\left(G^n_T,\ \bar K\right)\right]\vspace{0.3cm}\\
 &\Dis + \frac{\bar K }{(1-\theta)^{\alpha^*-2\over \alpha^*}}\left(\frac{12}{\eps}\right)^{2\over \alpha^*}\E\left[\psi\left(H^n (\theta),\ \bar K\right)\psi\left(G^n_T,\ \bar K\right)\int_0^T \left| g^n(s,Y_s, Z_s)-g(s,Y_s, Z_s)\right| {\rm d}s \right].\vspace{0.1cm}
\end{array}
$$
Observe from the inequalities \eqref{eq:6.1} and \eqref{eq:6.6} that the sequences
$$
\left(\sup_{t\in \T}(|Y^n_t|+|Y_t|)\right)_{n\geq 1}, \ \ \left(\psi\left(G^n_T,\ \bar K\right)\right)_{n\geq 1}\ \  {\rm and}\ \  \left(\psi\left(H^n (\theta),\ \bar K\right)\right)_{n\geq 1}
$$
are bounded in all $L^p$ spaces. From the previous inequality together with H\"{o}lder's inequality we deduce that there exist a universal constant $C>0$ and a constant $C(\theta,q)>0$ depending only on $\theta$ and $q$ such that
\begin{equation}\label{eq:6.14}
\begin{array}{ll}
&\Dis \mathbb{P}\left(\sup_{t\in\T} |Y^n_t-Y_t|>\eps\right)\vspace{0.3cm}\\
\leq &\Dis \frac{3(1-\theta)}{\eps}C+\left(\frac{12(1-\theta)}{\eps}\right)^{2\over \alpha^*}\bar K C \left(\E\left[\psi\left(X^n (\theta),\ 2\bar K\right)\right]\right)^{1/2}\vspace{0.3cm}\\
&\Dis + \frac{\bar K C(\theta,q) }{(1-\theta)^{\alpha^*-2\over \alpha^*}}\left(\frac{12}{\eps}\right)^{2\over \alpha^*}\left(\E\left[\left(\int_0^T \left| g^n(s,Y_s, Z_s)-g(s,Y_s, Z_s)\right| {\rm d}s\right)^q\right]\right)^{1\over q}.\vspace{0.2cm}
\end{array}
\end{equation}
From inequality \eqref{eq:6.1} and the fact that $\ps$, $\xi^n\To \xi$, it follows
that as $n$ goes to $\infty$, $\E\left[\psi\left(X^n (\theta),\ 2\bar K\right)\right]$ converges to $\E\left[\psi\left(|\xi|,\ 2\bar K\right)\right]$ for each $\theta\in (0,1)$. Thus, in view of \eqref{eq:6.2}, sending first $n\To \infty$ and then $\theta\To 1$ in inequality \eqref{eq:6.14} yields that $\sup_{t\in\T}|Y^n_t-Y_t|$ converges to $0$ in probability, and the conclusion \eqref{eq:6.3} follows due to the inequality \eqref{eq:6.6}. \vspace{0.2cm}

Finally, let \eqref{eq:6.4} be further satisfied, and we show that \eqref{eq:6.5} holds. In fact, by It\^{o}'s formula we get that for each $n\geq 1$,
\begin{equation}\label{eq:6.15}
\begin{array}{ll}
&\Dis \E\left[\int_0^T |Z^n_s-Z_s|^2 {\rm d}s\right]\vspace{0.3cm}\\
\leq &\Dis \E\left[|\xi^n-\xi|^2+2\sup_{t\in\T}|Y^n_t-Y_t| \int_0^T \left| g^n(s,Y^n_s, Z^n_s)-g(s,Y_s, Z_s)\right| {\rm d}s \right].
\end{array}
\end{equation}
On the other hand, it follows from \ref{H1"} with \eqref{eq:6.4} as well as \eqref{eq:6.1} and \eqref{eq:6.6} that
\begin{equation}\label{eq:6.16}
\sup_{n\geq 1}\E\left[\left(\int_0^T \left| g^n(s,Y^n_s, Z^n_s)-g(s,Y_s, Z_s)\right| {\rm d}s\right)^2\right]<+\infty.
\end{equation}
Then, by virtue of H\"{o}lder's inequality, the desired conclusion \eqref{eq:6.5} follows from \eqref{eq:6.15}, \eqref{eq:6.16}, \eqref{eq:6.3} and \eqref{eq:6.6}. The proof is then complete.
\end{proof}

\begin{rmk}\label{rmk:6.2}
We note that if assumption \ref{H1"} for $g$ and $g^n$ in \cref{thm:6.1} is respectively replaced with the stronger assumption \ref{H1} with the process $f_t$ instead of $|g(t,0,0)|$ and $|g^n(t,0,0)|$, and assumption \eqref{eq:6.2} is replaced with the following assumption: $\as$, for each $(y,z)\in \R\times\R^d$, $g^n(t,y,z)\To g(t,y,z)$, then the conclusions \eqref{eq:6.3}  and \eqref{eq:6.5} of \cref{thm:6.1} still hold. Indeed, it is easy to check that these two assumptions together with \eqref{eq:6.1} can imply that \eqref{eq:6.2} holds for any $q>1$, and that \eqref{eq:6.4} holds for some $c>0$.
\end{rmk}

\section{Application to sub-quadratic PDEs}
\label{sec:7}
\setcounter{equation}{0}

In this section, we give an application of our results concerning BSDEs to PDEs which are sub-quadratic with respect to the gradient of the solution. More precisely, we will derive the nonlinear Feynman-Kac formula for these PDEs. Let us consider the following semilinear PDE
\begin{equation}\label{eq:7.1}
\partial_t u(t,x)+\mathcal{L} u(t,x)+g(t,x,u(t,x),\sigma^* \nabla_x u(t,x))=0,\ \ \ u(T,\cdot)=h(\cdot),
\end{equation}
where $\mathcal{L}$ is the infinitesimal generator of the diffusion solution $X^{t,x}_\cdot$ to the following SDE
\begin{equation}\label{eq:7.2}
X^{t,x}_s=x+\int_t^s b(r,X^{t,x}_r){\rm d}r+\int_t^s \sigma(r,X^{t,x}_r){\rm d}B_r,\ \ t\leq s\leq T,\ \ {\rm and}\ \ X^{t,x}_s=x,\ \ 0\leq s< t.
\end{equation}
The nonlinear Feynman-Kac formula consists in proving that the function defined by
\begin{equation}\label{eq:7.3}
\RE\ (t,x)\in \T\times \R^n,\ \ \ u(t,x):=Y^{t,x}_t,
\end{equation}
where, for each $(t_0,x_0)\in \T\times \R^n$, $(Y^{t_0,x_0}_\cdot,Z^{t_0,x_0}_\cdot)$ represents the solution to the BSDE
\begin{equation}\label{eq:7.4}
Y_t=h\left(X^{t_0,x_0}_T\right)+\int_t^T g(s,X^{t_0,x_0}_s,Y_s,Z_s){\rm d}s+\int_t^T Z_s\cdot {\rm d}B_s,\ \ t\in \T,
\end{equation}
is a solution, at least a viscosity solution, to PDE \eqref{eq:7.1}.

The objective of this section is to derive
the above probabilistic representation for the solution to PDE \eqref{eq:7.1} when the nonlinearity $g$ is sub-quadratic of $\alpha\in (1,2)$ order with respect to $\nabla_x u$ and when $h$ and $g$ have a power growth of $p<2/\alpha^*$ order with respect to $x$. Let us first recall the following definition of a viscosity solution to PDE \eqref{eq:7.1}.

\begin{dfn}\label{def:7.1}
A continuous function $u$ defined on $\T\times \R^n$ such that $u(T,\cdot)=h(\cdot)$ is said to be a viscosity super-solution (respectively sub-solution) to PDE \eqref{eq:7.1} if
$$
\partial_t u(t_0,x_0)+\mathcal{L} u(t_0,x_0)+g(t_0,x_0,u(t_0,x_0),\sigma^* \nabla_x \varphi(t_0,x_0))\leq 0\ \ \ ({\rm resp.}\ \ \geq 0)
$$
as soon as the function $u-\varphi$ has a local minimum (resp. maximum) at the point $(t_0,x_0)\in (0,T)\times \R^n$ where $\varphi$ is a smooth function. Moreover, a viscosity solution is both a viscosity super-solution and a viscosity sub-solution.
\end{dfn}

Let us now introduce our assumptions concerning the linear part of the PDE namely the coefficients of the diffusion.
\begin{enumerate}
\renewcommand{\theenumi}{(A1)}
\renewcommand{\labelenumi}{\theenumi}
\item\label{A1}
$b(t,x):\T\times \R^n\To \R^n$ and $\sigma(t,x):\T\times \R^n\To \R^{n\times d}$ are continuous functions and there exists a constant $K>0$ such that for each $(t,x,x')\in \T\times\R^n\times\R^n$,
$$
|b(t,0)|+|\sigma(t,x)|\leq K\vspace{-0.2cm}
$$
and
$$
|b(t,x)-b(t,x')|+|\sigma(t,x)-\sigma(t,x')|\leq K|x-x'|.\vspace{0.1cm}
$$
\end{enumerate}

Classical results on SDEs show that under the assumption \ref{A1}, for each $(t,x)\in \T\times \R^n$ the SDE \eqref{eq:7.2} admits a unique solution $X^{t,x}_\cdot\in \s^q$ for all $q\geq 1$. And, since $\sigma$ is  a bounded function, an argument in page 563 of \citet{BriandHu2008PTRF} has showed that for each $q\in [1,2)$ and $(t,x)\in \T\times \R^n$, we have
\begin{equation}\label{eq:7.5}
\RE\ \lambda>0,\ \ \ \E\left[\sup_{s\in [0,T]} \exp\left(\lambda |X^{t,x}_s|^q\right)\right]\leq C \exp(\lambda C |x|^q),\vspace{0.1cm}
\end{equation}
where the constant $C$ depends only on $(q,\lambda,T,K)$. Furthermore, we assume that the point sequence $\{(t_m,x_m)\}_{m=1}^\infty$ in the space $\T\times\R^n$ converges to a point $(t,x)\in \T\times\R^n$ as $m$ tends to $+\infty$. Classical results on SDEs show that
\begin{equation}\label{eq:7.6}
\RE\ \lambda>0,\ \ \ \lim\limits_{m\To \infty} \E\left[\sup\limits_{s\in \T}|X^{t_m,x_m}_s-X^{t,x}_s|^\lambda\right]=0.
\end{equation}
And, by a similar analysis as that in page 563 of \cite{BriandHu2008PTRF} we can also deduce that for each $q\in [1,2)$,
\begin{equation}\label{eq:7.7}
\RE\ \lambda>0,\ \ \ \E\left[\sup_{m\geq 1}\sup_{s\in [0,T]} \exp\left(\lambda |X^{t_m,x_m}_s|^q\right)\right]\leq C \exp(\lambda C |x|^q),
\end{equation}
where the constant $C$ is the same as in \eqref{eq:7.5}.

With these observations in hands, we can give our assumptions on the nonlinear term of the PDE, the generator $g$, and the terminal condition.
\begin{enumerate}
\renewcommand{\theenumi}{(A2)}
\renewcommand{\labelenumi}{\theenumi}
\item\label{A2} $g(t,x,y,z):\T\times \R^n\times\R\times\R^d\To \R$ and $h(x):\R^n\To \R$ are continuous functions and there exist three real constants $\alpha\in (1,2)$, $p\in [1,\alpha^*)$ and $k\geq 0$ such that for each $(t,x,y,z)\in \T\times \R^n\times \R\times\R^d$, $(y',z')\in \R\times\R^d$ and $\theta\in (0,1)$, it holds that
$$
{\rm sgn}(y)g(t,x,y,z)\leq k \left(1+|x|^p+|y|+|z|^\alpha\right),
$$
$$
|g(t,x,y,z)|+|h(x)|\leq k\left(1+|x|^p+\exp(k |y|^{2\over \alpha^*})+|z|^2\right)
$$
and
$$
  \begin{array}{ll}
  &\Dis {\bf 1}_{\{y-\theta y'>0\}}\left(g(t,x,y,z)-\theta g(t,x, y',z')\right)\vspace{0.2cm}\\
   \leq &\Dis (1-\theta) k \left(1+|x|^p+|y'|+\left|\frac{y-\theta y'}{1-\theta}\right|+\left|\frac{z-\theta z'}{1-\theta}\right|^\alpha\right)
   \end{array}
$$
   or
$$
  \begin{array}{ll}
  &\Dis -{\bf 1}_{\{y-\theta y'<0\}}\left(g(t,x,y,z)-\theta g(t,x, y',z')\right)\vspace{0.2cm}\\
   \leq &\Dis (1-\theta) k\left(1+|x|^p+|y'|+\left|\frac{y-\theta y'}{1-\theta}\right|+\left|\frac{z-\theta z'}{1-\theta}\right|^\alpha\right).\vspace{0.3cm}
   \end{array}
$$
\end{enumerate}

The following example shows that assumption \ref{A2} are more general than those used in some existing literature.

\begin{ex}
From \cref{pro:5.1} and \cref{rmk:5.2}, it is not difficult to verify that the assumption \ref{A2} holds for the following generator $g$ and  terminal function $h$:
$$
\begin{array}{l}
g(t,x,y,z):=1+|x|^p\sin|x|+y^{2m}{\rm 1}_{y\leq 0}+\sin y+|z|^{\alpha}+l(y)q(z),\vspace{0.2cm}\\
h(x):=|x|^p\cos|x|,\ \ \ (t,x,y,z)\in \T\times \R^n\times \R\times\R^d,
\end{array}\vspace{0.1cm}
$$
where $\alpha\in (1,2)$, $p\in [1,\alpha^*)$, $m$ is a positive integer, and the function $l(\cdot)$ and $q(\cdot)$ are defined in (iii) of \cref{pro:5.1}. It is clear that $g$ has a power growth in the state variables $(y,z)$ and it is non-Lipschitz continuous in $y$ and non-convex (non-concave) in $z$, and that both $g$ and $h$ have a power growth in the state variable $x$ and they are not uniformly continuous in $x$.
\end{ex}

In the sequel, due to $p\in [1,\alpha^*)$, it follows from \eqref{eq:7.5} that for each $(t_0,x_0)\in \T\times\R^n$ and each $\lambda>0$, we have
\begin{equation}\label{eq:7.8}
\Dis \E\left[\exp\left(\lambda \left(|X^{t_0,x_0}_T|^p\right)^{2\over \alpha^*}\right)\right] \leq \Dis  \E\left[\sup_{s\in [0,T]} \exp\left(\lambda |X^{t_0,x_0}_s|^{2p\over \alpha^*}\right)\right]\leq  \Dis \bar C \exp(\lambda \bar C |x_0|^{2p\over \alpha^*})<+\infty
\end{equation}
and
\begin{equation}\label{eq:7.9}
\begin{array}{lll}
\Dis \E\left[\exp\left(\lambda \left(\int_0^T |X^{t_0,x_0}_s|^p{\rm d}s\right)^{2\over \alpha^*}\right)\right] &\leq & \Dis  \E\left[\sup_{s\in [0,T]} \exp\left(\lambda T^{2\over \alpha^*} |X^{t_0,x_0}_s|^{2p\over \alpha^*}\right)\right]\vspace{0.2cm}\\
&\leq & \Dis \bar C \exp(\lambda T^{2\over \alpha^*}\bar C |x_0|^{2p\over \alpha^*})<+\infty,\vspace{0.1cm}
\end{array}
\end{equation}
where the constant $\bar C$ depends only on $(p,\alpha,\lambda,T,K)$. Then, the assumption \ref{A2} together with the inequalities \eqref{eq:7.8} and \eqref{eq:7.9} allows us to use \cref{thm:5.5} to construct a unique solution, $(Y^{t_0,x_0}_\cdot,Z^{t_0,x_0}_\cdot)$, to the BSDE \eqref{eq:7.4} such that
$$
\RE\ \lambda>0,\ \ \ \E\left[\exp\left\{\lambda\left(\sup_{t\in \T}|Y^{t_0,x_0}_t|\right)^{2\over \alpha^*}\right\}\right]<+\infty
$$
and $Z^{t_0,x_0}_\cdot\in \mcal^q$ for all $q\geq 1$. Furthermore, by a classical analysis we know that $u$ defined by the formula \eqref{eq:7.3} is a deterministic function.

Now we can state and prove the main result of this section.

\begin{thm}\label{thm:7.3}
Let the assumptions \ref{A1} and \ref{A2} hold. Then, the function $u$ defined in \eqref{eq:7.3} is continuous on $\T\times\R^n$ and there exists a constant $C>0$ such that
\begin{equation}\label{eq:7.10}
\RE\ (t,x)\in \T\times\R^n,\ \ \ |u(t,x)|\leq C(1+|x|^p).
\end{equation}
Moreover, $u$ is a viscosity solution to PDE \eqref{eq:7.1}.
\end{thm}

\begin{proof}
Let us first show that $u$ is a continuous function. Indeed, we assume that the point sequence $\{(t_m,x_m)\}_{m=1}^\infty$ in the space $\T\times\R^n$ converges to a point $(t,x)\in \T\times\R^n$ as $m$ goes to $+\infty$. From the continuity of function $h$ and inequality \eqref{eq:7.6} it follows that $\ps$,
\begin{equation}\label{eq:7.11}
\lim_{m\To\infty} h\left(X^{t_m,x_m}_T\right)=h\left(X^{t,x}_T\right).
\end{equation}
And, since $g$ is a continuous function and satisfies assumption \ref{A2}, by Lebesgue's dominated convergence theorem and inequalities \eqref{eq:7.6} and \eqref{eq:7.7} together with the integrability condition of the process $(Y^{t,x}_\cdot,Z^{t,x}_\cdot)$ we can derive that for each $q>1$,
\begin{equation}\label{eq:7.12}
\lim_{m\To\infty} \E\left[\left(\int_0^T \left|g(s,X^{t_m,x_m}_s,Y^{t,x}_s, Z^{t,x}_s)-g(s,X^{t,x}_s,Y^{t,x}_s,Z^{t,x}_s)\right|{\rm d}s\right)^q\right]=0.
\end{equation}
Furthermore, in view of the growth condition of function $h$ and inequality \eqref{eq:7.7}, a similar argument to \eqref{eq:7.8} and \eqref{eq:7.9} yields that for each $\lambda>0$,
\begin{equation}\label{eq:7.13}
\sup_{m\geq 1}\E\left[\exp\left\{\lambda \left(\left|h\left(X^{t_m,x_m}_T\right)\right|+\int_0^T |X^{t_m,x_m}_s|^p{\rm d}s\right)^{\frac{2}{\alpha^*}} \right\}\right]<+\infty.
\end{equation}
In view of \eqref{eq:7.11}-\eqref{eq:7.13} and \ref{A2}, using the stability theorem (\cref{thm:6.1}) leads to $\ps$,
$$
\lim_{m\To\infty} \sup_{s\in\T}\left|Y^{t_m,x_m}_s-Y^{t,x}_s\right|=0,
$$
which together with the continuity of $Y^{t,x}_\cdot$ with respect to the time variable yields that $u$ is a continuous function on $\T\times\R^n$.\vspace{0.2cm}

Secondly, in view of assumption \ref{A2} and \cref{rmk:3.5}, the inequality \eqref{eq:7.10} follows from inequalities \eqref{eq:7.8} and \eqref{eq:7.9} with the estimates \eqref{eq:3.4} and \eqref{eq:3.5} in \cref{thm:3.1}.\vspace{0.2cm}

Finally, we use a double approximation procedure and a stability result to prove that the function $u$ is a viscosity solution to PDE \eqref{eq:7.1}. For each $(t,x,y,z)\in \T\times\R^n\times\R\times\R^d$ and each pair of positive integers $m$ and $l$, we define
$$
\begin{array}{lll}
g^{m,l}(t,x,y,z)&:=& \Dis \inf\left\{g^+(t,x,y,z)+m|y-y'|+m|z-z'|,\ \ (y',z')\in \mathbb{Q}\times \mathbb{Q}^d \right\}\vspace{0.2cm}\\
&& \Dis -\inf\left\{g^-(t,x,y,z)+l|y-y'|+l|z-z'|,\ \ (y',z')\in \mathbb{Q}\times \mathbb{Q}^d \right\}.
\end{array}
$$
By \citet{LepeltierSanMartin1997SPL} it is well known that $g^{m,l}$ is uniformly Lipschitz continuous in $(y,z)$, $g^{m,l}$ converges decreasingly uniformly on compact sets to a limit $g^{m,\infty}$ as $l$ tends to $+\infty$, and $g^{m,\infty}$ converges increasingly uniformly on compact sets to the generator $g$ as $m$ tends to $+\infty$.
For $(t,x)\in\T\times \R^n$, let $(Y^{m,l,t,x}_\cdot, Z^{m,l,t,x}_\cdot)$ be the unique solution in $\s^2\times \mcal^2$ to the BSDE with the terminal condition $h\left(X^{t,x}_T\right)$ and the generator
$g^{m,l}\left(\cdot, X^{t,x}_\cdot, \cdot,\cdot\right)$. We denote $u^{m,l}(t,x):=Y^{m,l,t,x}_t$. Then, by the classical nonlinear Feynman-Kac formula (see, e.g. \citet{ElKarouiPengQuenez1997MF} and \citet{PardouxPeng1992LectureNotes}), $u^{m,l}(\cdot,\cdot)$ is a viscosity solution to the following PDE
$$
\partial_t u(t,x)+\mathcal{L} u(t,x)+g^{m,l}(t,x,u(t,x),\sigma^* \nabla_x u(t,x))=0,\ \ \ u(T,\cdot)=h(\cdot).
$$
Moreover, by virtue of the classical comparison theorem and the stability \cref{thm:6.1} we can derive that $u^{m,l}(\cdot,\cdot)$ is decreasing and converges pointwisely to a continuous function $u^{m,\infty}(\cdot,\cdot)$ as $l$ tends to $+\infty$, and $u^{m,\infty}(\cdot,\cdot)$ is increasing and converges pointwisely to the continuous function $u(\cdot,\cdot)$ as $m$ tends to $+\infty$. Dini's theorem implies that the convergence is also uniform on compact sets of $\T\times\R^n$. Then, we can apply the stability theorem 1.7 in Chapter 5 of \citet{BardiCapuzzo-Dolcetta1997Book} to show that $u$ is a viscosity solution to the PDE \eqref{eq:7.1}. The proof is then complete.
\end{proof}

\begin{rmk}\label{rmk:7.4}
When the generator $g$ does not depend on the variable $y$ and is convex or concave on the variable $z$, it can be shown that the function $u$ defined by the formula \eqref{eq:7.3} is the unique viscosity solution with following growth: $|u(t,x)|\leq C(1+|x|^{\alpha^*})$. This follows from
the uniqueness results in \citet{DaLioLey2006SIAM} concerning Bellman-Isaacs equation.
\end{rmk}

\vspace{0.2cm}



\setlength{\bibsep}{2pt}
\bibliographystyle{elsarticle-harv}

\end{document}